\documentclass{amsart}
\usepackage[T1]{fontenc}
\usepackage{amsmath}
\usepackage{amssymb,bbm}
\usepackage{amsthm}
\usepackage{enumerate}
\usepackage[utf8]{inputenc}

\usepackage{a4wide}
\usepackage{hyperref}
\usepackage[yyyymmdd,hhmmss]{datetime}

\usepackage{mathtools}

\newcommand{\G}{\ensuremath {\mathcal{G}} }
\newcommand{\Go}{\ensuremath {\mathcal{G}^{(0)}} }

\newcommand{\Ga}{\ensuremath {\mathcal{G}^{a}} }
\newcommand{\Lc}{\ensuremath {\mathcal{L}_c(X)} }
\newcommand{\Lco}{\ensuremath {\mathcal{L}_c(\Go)} }

\newcommand{\Lgo}{\ensuremath {\mathcal{L}_c(\Go)\rtimes \Ga} }
\newcommand{\Skew}{\mathcal{A} \rtimes_\pi S}
\newcommand{\A}{\mathcal{A}}
\newcommand{\J}{\mathcal{J}}
\newcommand{\F}{\mathcal{F}}

\DeclareMathOperator{\supp}{supp}
\DeclareMathOperator{\id}{id}
\DeclareMathOperator{\mini}{min}
\DeclareMathOperator{\Iso}{Iso}
\DeclareMathOperator{\Span}{Span}
\DeclareMathOperator{\Identity}{id}
\DeclareMathOperator{\Int}{int}

\theoremstyle{plain}
\newtheorem{theorem}{Theorem}[section]
\newtheorem{lemma}[theorem]{Lemma}
\newtheorem{prop}[theorem]{Proposition}
\newtheorem{corollary}[theorem]{Corollary}
\newtheorem{definition}[theorem]{Definition}
\newtheorem{remark}[theorem]{Remark}
\newtheorem{example}[theorem]{Example}

\begin{document}

\title[Simplicity of skew inverse semigroup rings]{Simplicity of skew inverse semigroup rings \\ with applications to Steinberg algebras \\ and topological dynamics}

\author{Viviane Beuter}
\address{Departamento de Matem\'{a}tica, Universidade Federal de Santa Catarina, Florian\'{o}polis, BR-88040-900, Brasil
and
Departamento de Matem\'{a}tica, Universidade do Estado de Santa Catarina, Joinville, BR-89219-710, Brasil}
\email{vivibeuter@gmail.com}

\author{Daniel Gon\c{c}alves}
\address{Departamento de Matem\'{a}tica, Universidade Federal de Santa
Catarina, Florian\'{o}polis, BR-88040-900, Brasil}
\email{daemig@gmail.com}

\author{Johan \"{O}inert}
\address{Department of Mathematics and Natural Sciences, Blekinge Institute of Technology,
SE-37179 Karlskrona, Sweden}
\email{johan.oinert@bth.se}

\author{Danilo Royer}
\address{Departamento de Matem\'{a}tica, Universidade Federal de Santa
Catarina, Florian\'{o}polis, BR-88040-900, Brasil}
\email{daniloroyer@gmail.com}



\date{\today}

\begin{abstract}
Given a partial action $\pi$ of an inverse semigroup $S$ on a ring $\A$ one may construct its associated skew inverse semigroup ring $\Skew$.
Our main result asserts that, when $\A$ is commutative, the ring $\Skew$ is simple if, and only if, $\A$ is a maximal commutative subring of $\Skew$ and $\A$ is $S$-simple. We apply this result in the context of topological inverse semigroup actions to connect simplicity of the associated skew inverse semigroup ring with topological properties of the action. Furthermore, we use our result to present a new proof of the simplicity criterion for a Steinberg algebra $A_R(\G)$ associated with a Hausdorff and ample groupoid $\G$.
\end{abstract}

\maketitle

\section{Introduction}\label{Sec:Intro}

The notion of a partial action of a group on a C*-algebra, and the construction of its associated crossed product C*-algebra (initially introduced by Exel \cite{exel1994a}), is a key ingredient in the study of many C*-algebras, e.g. Cuntz-Krieger algebras \cite{exellaca}, Cuntz-Li algebras \cite{boavaexel}, graph C*-algebras \cite{toke}, ultragraph C*-algebras \cite{GRUltra, GRIRMN}, and algebras associated with Bratteli diagrams \cite{GGS, GRHouston}, to name a few.

In a purely algebraic context, partial skew group rings were introduced by Dokuchaev and Exel \cite{dokuchaevexel} as a generalization of classical skew group rings and as an algebraic analogue of partial crossed product C*-algebras. The theory of partial skew group rings, which is still quite young, is less developed than its analytical counterpart, but it has evolved quickly in recent years. In particular many important algebras, such as Leavitt path algebras \cite{goncalvesroyer} and ultragraph Leavitt path algebras \cite{GRUP}, can be realized as partial skew group rings and general results about the theory, as the ones in \cite{beuter1,Dokuchaev,Gonc,GOR,OinertArt}, have been applied to study these algebras (see \cite{Misha} for a comprehensive overview of developments in the theory of partial actions). This recent development of the area indicates that the theory of non-commutative rings may benefit from the theory of partial skew group rings.

In this article we shall be concerned with \emph{skew inverse semigroup rings}. This class of rings was introduced by Exel and Vieira (see e.g. \cite{beuter2,exelvieira}) and generalizes the class of partial skew group rings (see \cite[Theorem 3.7]{exelvieira}). Our interest to study this class of rings comes from its connections with topological dynamics (see Section~\ref{topd}), and the fact that any Steinberg algebra, associated with a Hausdorff and ample groupoid, can be realized as a skew inverse semigroup ring (see \cite{beuter2}). Recall that Steinberg algebras were introduced by Steinberg in \cite{steinberg} and, independently, by Clark et al. in \cite{clarketal} (see Section~\ref{Sec:Steinberg} for the definition). They are ''algebraisations'' of Renault’s C*-algebras of groupoids. Lately, Steinberg algebras have attracted a lot of attention, partly since they include all Kumjian-Pask algebras of higher-rank graphs introduced in \cite{arandapinoetal} and therefore also all Leavitt path algebras. For some examples of the theory of Steinberg algebras, we refer the reader to \cite{clarkediemichell,clarkexelpardo,clarksims,steinberg2014}. 

The interplay between topological dynamics and crossed products algebras is a driving force in the field of C*-algebras and has motivated the study of relations between topological dynamics and purely algebraic objects (as Steinberg algebras). By applying our main results we can describe connections between simplicity of the skew inverse semigroup ring associated with a topological partial action and topological properties of the action. The techniques we employ here are quite different from the ones used in \cite{GOR}.

This article is organized as follows. In Section~\ref{Sec:background}, we recall some important properties of partial actions of inverse semigroups and the construction of the skew inverse semigroup ring. For the latter a quotient by an ideal, which we denote by $\mathcal{N}$, is necessary. We provide a description of this ideal in Lemma~\ref{Idescription}.
In Section~\ref{Sec:Simplicity}, 
we define the \emph{diagonal} of $\Skew$, denoted by $\mathcal{D}$, which is isomorphic to $\A$ (see Proposition~\ref{prop:DisoA}).
Furthermore,
we prove our main result, which yields a complete characterization of simplicity of skew inverse semigroup rings in the case when $\A$ is commutative (see Theorem~\ref{thm:main1}). 
In Section~\ref{topd} we apply our result in the context of topological dynamics: Given an topological partial action of an inverse semigroup on a zero-dimensional locally compact Hausdorff space, we show that the associated skew inverse semigroup ring is simple if, and only if, the action is minimal, topologically principal and a certain condition on the existence of functions with non-empty support in ideals of the skew inverse semigroup ring holds. (The aforementioned condition has the same flavour as the one presented in \cite{galera} for groupoids. We were not aware of the work in \cite{galera} while developing Section~\ref{topd}.) 
Finally, in Section~\ref{Sec:Steinberg}, based on the above mentioned work of Beuter and Gon\c{c}alves \cite{beuter2}, we apply our main result to get a new proof of the simplicity criterion for a Steinberg algebra $A_R(\G)$ associated with a Hausdorff and ample groupoid $\G$ (see Theorem~\ref{thm:SteinbergSimplicity}).

\section{Preliminaries}\label{Sec:background}

\subsection{Inverse semigroups and their partial actions}\label{subsec:partialactions}

Throughout this article $S$ will denote an \emph{inverse semigroup}.
Recall that this means that $S$ is a semigroup such that for each element $s\in S$
there is a unique element $s^* \in S$ satisfying $ss^*s=s$ and $s^*ss^*=s$.
We denote by $E(S)=\{e \in S \mid e^2=e\}$ the set of idempotents of $S$. 
$E(S)$ is a commutative sub-inverse semigroup of $S$ and it is moreover a lattice under the partial order:
$e \leq f $ if and only if $e = ef$. This partial order can be extended to $S$ by putting
\begin{displaymath}
	s \leq t \quad \iff \quad s=t s^*s \quad \iff \quad s = ss^* t,
\end{displaymath}
for any $s,t\in S$.
Products and inverses in $S$ 
are preserved by this partial order
in the sense that
\begin{itemize}
\item if $s\leq t$ and $u \leq v$, then $su \leq tv$;
\item if $s\leq t$, then $s^*\leq t^*$.
\end{itemize}
It is easy to see that $es \leq s$ and $se \leq s,$ for any $s \in S$ and $e \in E(S).$ In particular, if $S$ is unital then $e\leq 1$, for any $e\in E(S)$.

We now introduce the type of partial action that we shall be concerned with in this article.

\begin{definition}\label{def:partialaction}
A \emph{partial action of an inverse semigroup $S$ on a ring $\A$} is a collection of ideals 
$\{D_s\}_{s\in S}$ of $\A$ and ring isomorphisms $\{\pi_s : D_{s^*} \to D_s \}_{s\in S}$ such that,
for any $s,t\in S$, the following three assertions hold:
\begin{enumerate}[{\rm (i)}]
	\item\label{defi} $\A$, viewed as an additive group, is generated by the set $\bigcup_{e \in E(S)} D_e;$
	\item\label{defii} $\pi_s( D_{s^*} \cap D_t ) = D_s \cap D_{st}$;
	\item\label{defiii} $\pi_{s}(\pi_{t}(x))=\pi_{st}(x)$, for all $x\in D_{t^*} \cap D_{t^*s^*}$.
\end{enumerate}
\end{definition}

The following proposition shows how Definition~\ref{def:partialaction}
is related to e.g. \cite[Definition 2.4]{beuter2} and \cite[Definition 3.3]{exelbuss}.

\begin{prop}\label{prop:partialaction}
Let $( \{ \pi_s \}_{s\in S}, \{D_s\}_{s\in S} )$ be a partial action of $S$ on a ring $\A$,
as defined in Definition~\ref{def:partialaction}.
The following assertions hold:
\begin{enumerate}[{\rm (a)}]
	\item\label{prop:sss} $D_s \subseteq D_{ss^*}$, for any $s\in S$;
	\item\label{prop:PIidempotent} $\pi_e = \id_{D_e}$, for any $e\in E(S)$;
	\item\label{prop:inverse} $\pi_s^{-1} = \pi_{s^*}$, for any $s\in S$;
	\item\label{prop:posetstructure} Let $s,t\in S$. If $s\leq t$, then $D_s \subseteq D_t$;
	\item\label{prop:restriction} Let $s,t\in S.$ If $s \leq t,$ then $\pi_s(x)=\pi_t(x),$ for all $x \in D_{s^*}.$
	\item\label{prop:X1} If $S$ is unital, then $D_1 = \A$ and $\pi_1=\id_{\A}$;	
\end{enumerate}
\end{prop}

\begin{proof}
\eqref{prop:sss}: From Definition~\ref{def:partialaction}\eqref{defii} we get that
\begin{displaymath}
	D_s = \pi_{s}(D_{s^*} \cap D_{s^*}) = D_s \cap D_{ss^*} \subseteq D_{ss^*}.
\end{displaymath}

\eqref{prop:PIidempotent}:
Notice that $e=e^*$ for each $e\in E(S)$.
Take any $x \in D_{e^*} \cap D_{(e^*)^2} = D_{e^*} = D_e$.
By Definition~\ref{def:partialaction}\eqref{defiii}
we get that
\begin{displaymath}
	\pi_e(\pi_e(x))=\pi_{e^2}(x) = \pi_e(x)
\end{displaymath}
which shows that $\pi_e = \id_{D_e}$.

\eqref{prop:inverse}:
Take any $x\in D_s$. By \eqref{prop:sss} we get that $x\in D_{ss^*} = D_{(ss^*)^*}$
and hence, by Definition~\ref{def:partialaction}\eqref{defiii}, we may write
\begin{displaymath}
	\pi_s ( \pi_{s^*} (x) ) = \pi_{ss^*}(x).
\end{displaymath}
Using the fact that $ss^* \in E(S)$ we get, by \eqref{prop:PIidempotent},
that $\pi_s ( \pi_{s^*} (x) ) = x$.
Analogously one may show that $\pi_{s^*} ( \pi_{s} (x) ) = \pi_{s^*s}(x)=x$, for all $x\in D_{s^*}$.
This shows that $\pi_s^{-1} = \pi_{s^*}$.

\eqref{prop:posetstructure}:
Take any $y \in D_s$. Put $x=\pi_{s}^{-1}(y) \in D_{s^*}$.
Using that $s \leq t$ we get that $s=ts^*s$.
By \eqref{prop:sss} we notice that $x \in D_{s^*} \cap D_{s^*s} = D_{(ts^*s)^*} \cap D_{(s^*s)^*}$.
Hence, using Definition~\ref{def:partialaction}\eqref{defiii}
we get that
\begin{displaymath}
	y= \pi_s(x) = \pi_{ts^*s}(x) = \pi_t( \pi_{s^*s}( x ) ) \in D_t.
\end{displaymath}
This shows that $D_s \subseteq D_t$.

\eqref{prop:restriction}: 
We have that  $s^* \leq t^*$ and by \eqref{prop:posetstructure} we get that $D_{s^*} \subseteq D_{t^*}.$ 
Moreover, $s=ts^*s$ and $s^*=s^*st^*.$ For every $x \in D_{s^*}$ we have that $x \in  D_{s^*s}\cap D_{s^*st^*}$.
Hence, by Definition~\ref{def:partialaction}\eqref{defiii}, we conclude that
\begin{displaymath}
   \pi_s(x)=\pi_{ts^*s}(x)=\pi_t(\pi_{s^*s}(x))=\pi_t(x).
\end{displaymath}

\eqref{prop:X1}: We know that if $S$ is unital, then $e\leq 1$ for every $e\in E(S).$ By \eqref{prop:sss} we get that 
$D_s \subseteq D_{ss^*} \subseteq D_1$. Hence, Definition~\ref{def:partialaction}\eqref{defi} yields that
$\A=D_1$ and by \eqref{prop:PIidempotent} we get that $\pi_1=\id_{\A}.$
\end{proof}

\subsection{The skew inverse semigroup ring}\label{subsec:partialskewinversesemigroupring}

Given a partial action $(\{ \pi_s \}_{s\in S}, \{D_s\}_{s\in S} )$ of $S$ on a ring $\A$,
the construction of the corresponding skew inverse semigroup ring is done in three steps.
\begin{enumerate}
	\item First we consider the set
\begin{displaymath}
	\mathcal{L} = \left\{ \sum_{s\in S}^{\text{finite}} a_s \delta_s \ \Bigr\lvert \ a_s \in D_s \right\}
\end{displaymath}
where $\delta_s$, for $s\in S$, is a formal symbol.
We equip $\mathcal{L}$ with component-wise addition and with a multiplication 
defined as the linear extension of the rule
\begin{displaymath}
	(a_s \delta_s)(b_t \delta_t) = \pi_{s}(\pi_{s^*}(a_s) b_t) \delta_{st}.
\end{displaymath}
If we assume that $\A$ is associative, which we will, then using the assumption that $\A$ and each $D_s$, $s\in S$, have local units\footnote{Recall that a ring $R$ is said to have \emph{local units} if, for every finite subset $\F$ of $R$, there exists an idempotent $f\in R$ such that $\F \subseteq fRf$. In this case, $x=fx=xf$ holds for each $x\in \F$ and the element $f$ will be referred to as a \emph{local unit} for the set $\F$.}  one can show that $\mathcal{L}$ is an associative ring (see \cite[Theorem 3.4]{exelvieira}).

	\item Then, we consider the ideal
	\begin{displaymath}
	\mathcal{N} = \langle a \delta_r - a \delta_s \mid r,s \in S, \ r \leq s \text{ and } a \in D_r \rangle,
\end{displaymath}
i.e. $\mathcal{N}$ is the ideal of $\mathcal{L}$ generated by all elements of the form $a\delta_r - a\delta_s$, where $r\leq s$ and $a \in D_r$. (Notice that $a\in D_s$, by Proposition~\ref{prop:partialaction}\eqref{prop:posetstructure}.)

	\item Finally, we define the corresponding
	\emph{skew inverse semigroup ring}, which we denote by $\Skew$, as the quotient ring $\mathcal{L}/\mathcal{N}$.
Elements of $\Skew$ will be written as $\overline{x}$, where $x  \in \mathcal{L}$.
\end{enumerate}

It is not difficult to see that the construction of $\mathcal{L}$ yields 
a ring which is a generalization of a partial skew group ring \cite{dokuchaevexel}.
But in fact, one can also show that even $\Skew$ is a 
generalization of a partial skew group ring.
Indeed, given a partial skew group ring $\A \star_\alpha G$ one may define 
\emph{Exel's semigroup} $S_G$ (see e.g. \cite{exel1998}) associated with the group
$G$ and construct a certain skew inverse semigroup ring $\A \rtimes_\pi S_G$
which is isomorphic to $\A \star_\alpha G$ (see \cite[Theorem 3.7]{exelvieira}).

Next we describe the ideal $\mathcal{N}$ which was used above to define $\Skew$.

\begin{lemma}\label{Idescription} 
The ideal $\mathcal{N}$ is equal to the additive group generated by the set
$\{a\delta_r-a\delta_s \mid r,s \in S, \ r\leq s \text{ and }a\in D_r\}$.
\end{lemma}

\begin{proof}
It is enough to show that for $r,s,t,u\in S$ with $r\leq s$, and $a\in D_r$, $b\in D_t$, $c\in D_u$, it holds that the elements $b\delta_t(a\delta_r-a\delta_s)$, $(a\delta_r-a\delta_s)c\delta_u$ and $b\delta_t(a\delta_r-a\delta_s)c\delta_u$ are all of the form $x\delta_v-x\delta_w$ for some $v,w\in S$ and $x\in D_v$, such that $v\leq w$. Notice that 
\begin{displaymath}
	b\delta_t(a\delta_r-a\delta_s)=b\delta_ta\delta_r-b\delta_ta\delta_s=\pi_t(\pi_{t^*}(b)a)\delta_{tr}-\pi_t(\pi_{t^*}(b)a)\delta_{ts}
\end{displaymath}
and, since $tr\leq ts$, we are done in this case.

In the next case we get
\begin{displaymath}
	(a\delta_r-a\delta_s)c\delta_u=\pi_r(\pi_{r^*}(a)c)\delta_{ru}-\pi_s(\pi_{s^*}(a)c)\delta_{su}.
\end{displaymath}
Using that $r\leq s$ we get that $r^*\leq s^*$ and $ru\leq su$.
By Proposition~\ref{prop:partialaction}\eqref{prop:restriction},
$\pi_{r^*}(a)=\pi_{s^*}(a)$ and $\pi_r(\pi_{r^*}(a)c)=\pi_r(\pi_{s^*}(a)c)=\pi_s(\pi_{s^*}(a)c)$ and hence the desired conclusion follows. 

Finally, notice that $b\delta_t(a\delta_r-a\delta_s)c\delta_u$ is of the form $(x\delta_{tr}-x\delta_{ts})c\delta_u$ by the first case, and now, by the second case, $(x\delta_{tr}-x\delta_{ts})c\delta_u$ has the desired form.
\end{proof}

\begin{remark}\label{rem:calc}
Let $s,t\in S$. Notice that if $s \leq t$ and $a \in D_s$, then $\overline{a\delta_s}=\overline{a\delta_t}.$
\end{remark}

Just as for the skew group rings, we may define an additive map $\tau : \mathcal{L} \to \A$ defined by
\begin{displaymath}
\tau \left( \sum_{s\in S} a_s \delta_s \right)= \sum_{s\in S} a_s.
\end{displaymath}

\begin{remark} \label{tau}
By Lemma~\ref{Idescription}, we have that $\tau(\mathcal{N})=\{0\}$ and hence we get a well-defined additive map $\widetilde{\tau}: \A \rtimes_\pi S\rightarrow \A$ defined by $\widetilde{\tau}(\overline{x})=\tau(x),$ for $x \in \mathcal{L}$.
\end{remark}

\begin{lemma}\label{intersection} 
Let $H\subseteq \mathcal{L}$ be the set of all homogeneous elements of $\mathcal{L}$, that is, $H=\{a_s\delta_s \mid s\in S \text{ and }a_s\in D_s\}$. Then $H\cap \mathcal{N}=\{0\}$.
\end{lemma}

\begin{proof}
Take $x\in H\cap \mathcal{N}$. Then $x=a_s\delta_s$ for some $s\in S$ and $a_s\in D_s$. By applying $\tau$ on $x$ we get that $a_s=\tau(a_s\delta_s)=\tau(x)=0$, where the last equality follows from Lemma~\ref{Idescription}. This shows that $x=0$.
\end{proof}

\section{Simplicity of skew inverse semigroup rings}\label{Sec:Simplicity}

Throughout this section we shall make the following assumption: Any given  partial action $( \{ \pi_s \}_{s\in S}, \{D_s\}_{s\in S} )$ of $S$ on a ring $\A$ has the property that $\A$ and each ideal $D_s$, for $s\in S$, have local units.

We define the \emph{diagonal} of $\Skew$ as the following set:
\begin{displaymath}
	\mathcal{D}= \left\{ \sum_{i=1}^n\overline{a_i \delta_{e_i}}\ \Big\lvert \ n\in \mathbb{Z}_+, \ e_i\in E(S), \ a_i \in D_{e_i} \right\}.
\end{displaymath}

\begin{prop}\label{prop:DisoA}
Let $\A$ be a commutative and associative ring. Then $\A$ is embedded in $\Skew$ and is isomorphic to $\mathcal{D}$, which is a commutative subring of $\Skew$.
\end{prop}

\begin{proof}
It is easy see that $\mathcal{D}$ is a subring of $\Skew$ because, by Proposition~\ref{prop:partialaction}\eqref{prop:PIidempotent},
\begin{displaymath}
	\left(\sum_{i=1}^n\overline{a_i\delta_{e_i}}\right)\left(\sum_{j=1}^m\overline{b_j\delta_{f_j}}\right)=\sum_{i=1}^{n} \sum_{j=1}^{m} \overline{a_ib_j\delta_{e_if_j}} \ \in \mathcal{D}.
\end{displaymath}
Commutativity of $\mathcal{D}$ follows from the fact that $\A$ and $E(S)$ are commutative.

Next we show that $\mathcal{D}$ is isomorphic to $\A$. Notice that, by Definition~\ref{def:partialaction}(\ref{defi}), given $a \in \A$ we can write $a=\sum_{i=1}^n a_{e_i}$, where $n\in \mathbb{Z}_+$, $e_i \in E(S)$, and $a_{e_i} \in D_{e_i}$ for each $i\in \{1,\ldots,n\}$. Let $\phi: \A  \rightarrow \mathcal{D}$ be the map defined by 
\begin{displaymath}
	\phi(a)=\sum_{i=1}^n\overline{a_{e_i}\delta_{e_i}},
\end{displaymath}
for $a=\sum_{i=1}^n a_{e_i} \in \A$. Clearly, $\phi$ is additive.

We prove by induction that $\phi$ is well-defined. More precisely, we will show that if $\sum_{i=1}^n a_{e_i}=0$ for $n\in \mathbb{Z}_+$, $e_i\in E(S)$, and $a_{e_i} \in D_{e_i}$ for $i \in \{1,\ldots,n\}$, then $\sum_{i=1}^n\overline{a_{e_i}\delta_{e_i}}=0$.

If $a_{e_1}=0$, then clearly $\overline{a_{e_1}\delta_{e_1}}=0$.
Let $n\in \mathbb{Z}_+$ be arbitrary. As our induction hypothesis, suppose that if $\sum_{i=1}^n a_{e_i}=0$, then $\sum_{i=1}^n\overline{a_{e_i}\delta_{e_i}}=0$.

Now, let $f \in E(S)$, $a_{f} \in D_f$, and suppose that $a_f+\sum_{i=1}^n a_{e_i}=0$. Let $u \in \A$ be a local unit for $a_f, a_{e_1}, \ldots, a_{e_n}$, and let $u_f \in D_f$ be a local unit for $a_f$. Then
\begin{align*}
	0 &= (u-u_f)\left(a_f+\sum_{i=1}^n a_{e_i} \right)=ua_f-u_fa_f+(u-u_f)\left(\sum_{i=1}^n a_{e_i} \right) \\
	&=a_f-a_f + (u-u_f)\left(\sum_{i=1}^n a_{e_i} \right)=\sum_{i=1}^n(u-u_f)a_{e_i}.
\end{align*}
By the induction hypothesis, we conclude that
$\sum_{i=1}^n\overline{(u-u_f)a_{e_i}\delta_{e_i}}=0$.
Using this,
together with Remark~\ref{rem:calc}
and the fact that
$fe_i \leq e_i$ and $fe_i \leq f$, and hence $D_{f e_i} = D_f \cap D_{e_i}$, for each $i \in \{1,\ldots,n\}$,
we get that
\begin{align*}
\sum_{i=1}^n\overline{a_{e_i}\delta_{e_i}} &=\sum_{i=1}^n\overline{u_fa_{e_i}\delta_{e_i}}
	= \sum_{i=1}^n\overline{u_fa_{e_i} \delta_{fe_i}} = \sum_{i=1}^n\overline{u_fa_{e_i}\delta_f} \\
	&= \overline{\left(\sum_{i=1}^n u_fa_{e_i} \right)\delta_{f}}=\overline{\left(u_f\sum_{i=1}^n a_{e_i} \right)\delta_{f}}=\overline{(u_f(-a_f))\delta_{f}}=\overline{-a_f\delta_f}.
\end{align*}
Therefore, 
\begin{displaymath}
	\overline{a\delta_f}+\sum_{i=1}^n\overline{a_{e_i}\delta_{e_i}}=0,
\end{displaymath}
proving that $\phi$ is well-defined.

Clearly, $\phi$ is onto and multiplicative (using Proposition~\ref{prop:partialaction}\eqref{prop:PIidempotent})
and thus a surjective ring morphism.
Now, consider the map $\widetilde{\tau}$ which was defined in Remark~\ref{tau}. Notice that
\begin{displaymath}
	\widetilde\tau\circ\phi\left(\sum_{i=1}^n a_{e_i}\right)=\widetilde\tau\left(\sum_{i=1}^n\overline{a_{e_i}\delta_{e_i}}\right)=\sum_{i=1}^na_{e_i},
\end{displaymath}
that is, $\widetilde\tau\circ\phi=\id_{\mathcal{A}}$, and hence $\phi$ is injective.
\end{proof}

\begin{remark}
Suppose that $S$ is unital, with identity element $1 \in S$. In this case, if $e \in E(S)$, then $e\leq 1$, and therefore for each
$a \in D_e$ we have $\overline{a\delta_e}=\overline{a\delta_1}$. Hence, 
$\overline{\A \delta_1} = \mathcal{D}.$
\end{remark}

It does not make sense to speak of the support-length of an element in the quotient ring $\Skew$.
However, given any element $a\in \Skew$ we may speak of the \emph{minimal support-length
of a representative of $a$}, i.e. an element $x\in \mathcal{L}$ such that $a=\overline{x}$.
We make the following definition.

\begin{definition}
For each non-zero $a\in \Skew$ we define the number
\begin{displaymath}
n(a)= \mini\left\{ |F| \quad \Bigr\lvert \quad a = \sum_{s \in F} \overline{a_s \delta_s}  \ \text{ and } \ a_s \neq 0 \text{ for all } s\in F \right\},	
\end{displaymath}
where $|F|$ denotes the cardinality of the finite set $F.$
\end{definition}

Our goal is to give a characterization of simplicity for skew inverse semigroup rings $\Skew$ in the case when $\A$ is commutative (see Theorem~\ref{thm:main1}).
By Proposition~\ref{prop:DisoA}, the ring $\A$ is isomorphic to $\mathcal{D}$ which is a subring of $\Skew$.
Therefore, we will identify $\A$ with $\mathcal{D}$ and use $\A$ and $\mathcal{D}$ interchangeably.

Recall that the \emph{centralizer} of a non-empty subset $M$ of a ring $R$, denoted by $C_R(M)$, is the set of all the elements of $R$ that commute with each element of $M$. If $C_R(M)=M$ holds, then $M$ is said to be \emph{maximal commutative in $R$}. Notice that a maximal commutative subring is necessarily commutative.

\begin{theorem}\label{teo1}
Let $\A$ be an associative and commutative ring.
Then $\A$ is a maximal commutative subring of $\Skew$ if, and only if,
 $\J \cap \A \neq \{0\}$ for each non-zero ideal $\J$ of $\Skew$. 
\end{theorem}

\begin{proof}
We first show the ''if'' statement.
To this end, suppose that
$\A \cong \mathcal{D}$ is not a maximal commutative subring of $\Skew$. 
We now wish to conclude that there is some non-zero ideal $\J$ of $\Skew$ such that $\J \cap \mathcal{D} = \{0\}$.

Let $c = \sum\limits_{s \in F} \overline{c_s \delta_s} \in (\Skew) \setminus \mathcal{D}$
be an element which commutes with all the elements of $\mathcal{D}$.
Since $c$ commutes with $\overline{a \delta_e}$, for each $e\in E(S)$ and $a\in D_e$, we get that
\begin{displaymath}
	\sum\limits_{s\in F}\overline{ac_s\delta_{es}}=\sum\limits_{s\in F}\overline{\pi_s(\pi_{s^*}(c_s)a)\delta_{se}},
\end{displaymath} 
and hence
\begin{displaymath}
	\sum\limits_{s\in F}ac_s\delta_{es}-\sum\limits_{s\in F}\pi_s(\pi_{s^*}(c_s)a)\delta_{se}\in \mathcal{N}.
\end{displaymath}
Using that $\tau(\mathcal{N})=\{0\}$ we get that 
\begin{equation}\label{eq:coefficientscommuting}
	\sum\limits_{s\in F}\left(ac_s-\pi_s(\pi_{s^*}(c_s)a)\right)=0.
\end{equation}
Notice that
$x:=\sum\limits_{s\in F}\overline{c_s \delta_{ss^*}} - \sum\limits_{s\in F}\overline{c_s \delta_s} \neq \overline{0}.$ 
Otherwise we would have $c=\sum\limits_{s\in F}\overline{c_s \delta_s} =\sum\limits_{s\in F}\overline{c_s \delta_{ss^*}} \in \mathcal{D}.$

Now, let $\J$ be the non-zero ideal of $\Skew$ generated by the element $x.$
Each element of $\J$ is a finite sum of elements of the form
$\overline{a_u\delta_u}x\overline{a_v\delta_v}$, $\overline{a_u\delta_u}x$
and $x\overline{a_v\delta_v}$ for $u,v\in S$ and $a_u \in D_u$, $b_v \in D_v$.
By Proposition~\ref{prop:partialaction}\eqref{prop:PIidempotent}
and the fact that $ss^* \in E(S)$ we notice that
\begin{align*}
	\overline{a_u\delta_u}x\overline{a_v\delta_v} &= \overline{a_u\delta_u}\left(\sum\limits_{s\in F}\overline{c_s\delta_{ss^*}}-\sum\limits_{s\in F}		\overline{c_s\delta_s}\right)\overline{a_v\delta_v} \\
	&= \sum\limits_{s\in F}\overline{\pi_u(\pi_{u^*}(a_u)c_sa_v)\delta_{uss^*v}}-\sum\limits_{s\in F}\overline{\pi_u(\pi_{u^*}(a_u)\pi_s(\pi_{s^*}(c_s)a_v))\delta_{usv}}, 
\end{align*}
and hence, by Equation~\eqref{eq:coefficientscommuting}, we get that
\begin{align*}
	\widetilde{\tau}\left(\overline{a_u\delta_u}x\overline{a_v\delta_v}\right) &=\sum\limits_{s\in F}\pi_u(\pi_{u^*}(a_u)c_sa_v)-\sum\limits_{s\in F}\pi_u(\pi_{u^*}(a_u)\pi_s(\pi_{s^*}(c_s)a_v)) \\
	& =\pi_u\left(\pi_{u^*}(a_u)\sum\limits_{s\in F}(c_sa_v-\pi_s(\pi_{s^*}(c_s)a_v))\right)=0.
\end{align*}
Analogously, one may show that $\widetilde{\tau}(\overline{a_u\delta_u}x)=0$ and $\widetilde{\tau}(x\overline{a_v\delta_v})=0$.
This shows that $\widetilde{\tau}(\J)=\{0\}$.
Take any $y \in \J \cap \mathcal{D}$.
Then $y = \sum_{i=1}^n\overline{a_{i} \delta_{e_i}}$, for some $n\in \mathbb{Z}_+$ and $e_i\in E(S)$, $a_i\in D_{e_i}$
for $i\in \{1,\ldots,n\}$.
Notice that 
\begin{displaymath}
	\sum_{i=1}^na_i = \tau\left(\sum_{i=1}^n a_i\delta_{e_i} \right) = \widetilde{\tau}(y)= 0.
\end{displaymath}
Hence $y=0$ (for the same reason that $\phi$ is well-defined in Proposition~\ref{prop:DisoA}).
We now conclude that $\J\cap \mathcal{D}=\{0\}$.

Now we show the ''only if'' statement.
Suppose that $\mathcal{D} \cong \A$ is a maximal commutative subring of $\Skew$.
Let $\J$ be a non-zero ideal of $\Skew$.
Take $x \in \J \setminus \{0\}$ such that $n(x)=\mini \{ n(y) \ \mid \ y \in \J \setminus \{0\}  \}$
and write $x= \sum\limits_{s \in F}\overline{x_s \delta_s}$, where $|F|=n(x).$ 
Choose some $h\in F,$ and let $1_h \in D_h$ be a local unit for $x_h$.
By Proposition~\ref{prop:partialaction}\eqref{prop:sss}, $1_h\in D_{hh^*}$.
Notice that
\begin{displaymath}
	\overline{1_h\delta_{hh^*}}x=\overline{x_h\delta_h}+\sum\limits_{s\in F\setminus\{h\}}\overline{1_hx_s\delta_{hh^*s}}.
\end{displaymath}
Using that $hh^*s\leq s$, for each $s \in S$, we get that $\overline{1_hx_s\delta_{hh^*s}}=\overline{1_hx_s\delta_s}$ and hence
\begin{displaymath}
	\overline{1_h\delta_{hh^*}}x=\overline{x_h\delta_h}+\sum\limits_{s\in F\setminus\{h\}}\overline{1_hx_s\delta_s}.
\end{displaymath}
Let
$y=x-\overline{1_h\delta_{hh^*}}x=\sum\limits_{s\in F\setminus\{h\}}\overline{(1_hx_s-x_s)\delta_s}$
and notice that $y \in \J$.
Using that $n(x)$ is minimal and $y\in \J$ we conclude that $y=0$.
Thus, we have that $\sum\limits_{s\in F\setminus\{h\}}\overline{1_hx_s\delta_s}=\sum\limits_{s\in F\setminus\{h\}}\overline{x_s\delta_s}$
and hence
\begin{displaymath}
	x=\overline{x_h\delta_h}+\sum\limits_{s\in F\setminus\{h\}}\overline{1_hx_s\delta_s}.
\end{displaymath}
In particular, $\overline{1_h\delta_{hh^*}} \ x=x\neq 0$ and, since $\overline{1_h\delta_{hh^*}} \ x=\overline{1_h\delta_h} \  \overline{\pi_{h^*}(1_h)\delta_{h^*}} \  x$, we have that
$\overline{\pi_{h^*}(1_h)\delta_{h^*}}\ x \neq 0$.
Let $z=\overline{\pi_{h^*}(1_h)\delta_{h^*}} \ x \in \J$ and notice that $z$ is non-zero and that
\begin{align*}
z = \overline{\pi_{h^*}(1_h)\delta_{h^*}}x & = \overline{\pi_{h^*}(x_h)\delta_{h^*h}}+\sum\limits_{s\in F\setminus\{h\}}\overline{\pi_{h^*}(1_h)\delta_{h^*}{x_s}\delta_s}	\\
  & = \overline{\pi_{h^*}(x_h)\delta_{h^*h}}+\sum\limits_{s\in F\setminus\{h\}}\overline{\pi_{h^*}(1_hx_s)\delta_{h^*s}}.
\end{align*}

Now, let $\overline{a \delta_e} \in \mathcal{D}$ be arbitrary and consider the element $p=\overline{a \delta_e}\cdot z - z\cdot \overline{a \delta_e} \in \J.$ 
We have that 
\begin{displaymath}
p=\overline{a\pi_{h^*}(x_h)\delta_{ehh^*}}+\sum\limits_{s\in F\setminus\{h\}}\overline{a\pi_{h^*}(1_hx_s)\delta_{eh^*s}}-\overline{\pi_{h^*}(x_h)a\delta_{hh^*e}}-\sum\limits_{s\in F\setminus\{h\}}\overline{\pi_{h^*s}(\pi_{s^*h}(\pi_{h^*}(1_hx_s))a)\delta_{h^*se}}.
\end{displaymath}
Since $\A$ and $E(S)$ are commutative, we have that
\begin{displaymath}
p  = \sum\limits_{s\in F\setminus\{h\}}\overline{a\pi_{h^*}(1_hx_s)\delta_{eh^*s}}-\sum\limits_{s\in F\setminus\{h\}}\overline{\pi_{h^*s}(\pi_{s^*h}(\pi_{h^*}(1_hx_s))a)\delta_{h^*se}}.
\end{displaymath} 
Using that $eh^*s \leq h^*s$ and $h^*se \leq h^*s,$ we have that
\begin{displaymath}
p = \sum\limits_{s\in F\setminus\{h\}}\overline{a\pi_{h^*}(1_hx_s)\delta_{h^*s}} -\sum\limits_{s\in F\setminus\{h\}}\overline{\pi_{h^*s}(\pi_{s^*h}(\pi_{h^*}(1_hx_s))a)\delta_{h^*s}}.
\end{displaymath}
Hence, $n(p) < n(x)$ and by the minimality of $n(x)$ we conclude that $p=0.$

But this implies that $\overline{a \delta_e} \cdot z = z \cdot \overline{a \delta_e}$. Therefore
\begin{displaymath}
	\sum_{i=1}^n\overline{ a_i \delta_{e_i}} \cdot z = z \cdot \sum_{i=1}^n\overline{ a_i \delta_{e_i}},
\end{displaymath}
for all $\sum_{i=1}^n \overline{a_i \delta_{e_i}} \in \mathcal{D}$. Since $\mathcal{D} \cong \A$ is maximal commutative, we get that $z\in \mathcal{D}$.
We conclude that $\J\cap \mathcal{D} \neq \{0\}$.
\end{proof}

\begin{corollary}\label{anelmaximal}
Let $\A$ be an associative and commutative ring. If $\Skew$ is simple, then $\A$ is a maximal commutative subring of $\Skew$.
\end{corollary}

Recall that an ideal $I$ of $\A$ is \emph{$S$-invariant} if $\pi_s(I\cap D_{s^*})\subseteq I$ holds for each $s\in S$. The ring $\A$ is said to be \emph{$S$-simple} if $\A$ has no non-zero $S$-invariant proper ideal.

\begin{prop}\label{anelssimple}
Let $\A$ be an associative ring. If $\Skew$ is simple, then $\A$ is $S$-simple.
\end{prop}

\begin{proof}
Let $I$ be a non-zero $S$-invariant ideal of $\A$. Define the set
\begin{displaymath}
	\mathcal{H}=\left\lbrace\sum_{s \in S} \overline{a_s\delta_s} \in \Skew \,\, \Big| \,\, a_s \in I \cap D_s, \, s \in S \right\rbrace.
\end{displaymath}

Notice that $\mathcal{H} \neq \{0\}$. Indeed, let $a \in I$ be non-zero and let $u \in \A$ be a local unit for $a.$ By Definition~\ref{def:partialaction}\eqref{defi} there are idempotents $e_1, \ldots, e_n \in E(S)$ such that $u=\sum_{i=1}^n u_i$, with $u_i \in D_{e_i}$ for $i\in \{1,\ldots,n\}.$
Clearly,
\begin{displaymath}
	0 \neq a = ua= \sum_{i=1}^nu_ia.
\end{displaymath}
Using that $I$ is an ideal of $\A$, we get that $u_ia \in I\cap D_{e_i}$ for $i\in \{1,\ldots,n\}$, and hence $\sum_{i=1}^n\overline{u_ia\delta_{e_i}} \in \mathcal{H}.$
Let $\phi$ denote the ring isomorphism from the proof of 
Proposition~\ref{prop:DisoA}. Using that $a \neq 0$, we get that
$\sum_{i=1}^n\overline{u_ia\delta_{e_i}}=\phi(\sum_{i=1}^n u_ia) = \phi(a) \neq 0$.

Moreover, $\mathcal{H}$ is a left ideal of $\Skew$. Indeed, if $\overline{a_r\delta_r} \in \Skew$ and $a_s \in I \cap D_s$ then $(\overline{a_r\delta_r})( \overline{ a_s \delta_s}) = \overline{ \pi_r(\pi_{r^*}(a_r)a_s)\delta_{rs}}.$  Since $I$ is $S$-invariant, $\pi_r(\pi_{r^*}(a_r)a_s) \in I $, and from the definition of a partial action we get that $\pi_r(\pi_{r^*}(a_r)a_s) \in D_{rs}$. Hence, $\overline{a_r\delta_r a_s \delta_s} \in \mathcal{H}.$

Similarly, $\mathcal{H}$ is a right ideal of $\Skew$ and hence, by the simplicity of $\Skew$, we obtain that $\mathcal{H}=\Skew.$ From the definition of $\mathcal{H}$ we immediately see that $\widetilde\tau(\mathcal{H}) \subseteq I$, and from what was done above, $\widetilde\tau(\mathcal{H})=\widetilde\tau(\Skew)=\A.$ Thus, $I=\A$ and therefore $\A$ is $S$-simple.
\end{proof}

We are now ready to state and prove the main result of this article.

\begin{theorem}\label{thm:main1}
If $\A$ is an associative and commutative ring, then the following two assertions are equivalent:
\begin{enumerate}[{\rm (i)}]
	\item\label{mainthm:1} The skew inverse semigroup ring $\Skew$ is simple;
	\item\label{mainthm:2} $\A$ is $S$-simple, and $\A\cong \mathcal{D}$ is a maximal commutative subring of $\Skew$.
\end{enumerate}
\end{theorem}

\begin{proof}
\eqref{mainthm:1}$\Rightarrow$\eqref{mainthm:2}: This follows from Corollary~\ref{anelmaximal} and Proposition~\ref{anelssimple}.

\eqref{mainthm:2}$\Rightarrow$\eqref{mainthm:1}: Let $\J$ be a non-zero ideal of $\Skew.$ By Theorem \ref{teo1}, $\J \cap \mathcal{D} \neq \{0\}.$

Put $\mathcal{K}=\J \cap \mathcal{D}$ and $\mathcal{K}'=\phi^{-1}(\mathcal{K})$,
where $\phi: \A  \rightarrow \mathcal{D}$ is the ring isomorphism from Proposition~\ref{prop:DisoA}.
Clearly, $\mathcal{K}'$ is a non-zero ideal of $\A$.
Now we show that $\mathcal{K}'$ is $S$-invariant.

Take an arbitrary $s\in S$ and an arbitrary $a_s \in \mathcal{K}' \cap D_s$. Pick a local unit $1_s$ for $a_s$ in $D_s.$ 
By the definition of $\A$ there are idempotents $e_1, \ldots, e_n \in S$, and elements $a_{e_i} \in D_{e_i}$, for $i\in \{1,\ldots,n\}$, such that $a_s=\sum_{i=1}^n a_{e_i}$ and $\phi(a_s)=\sum_{i=1}^n\overline{a_{e_i} \delta_{e_i}} \in \mathcal{K}$. We notice that
\begin{align*}
\overline{\pi_{s^*}(1_s)\delta_{s^*}} \cdot \sum_{i=1}^n\overline{a_{e_i}\delta_{e_i}} \cdot \overline{1_s\delta_s}
                 & = \overline{\pi_{s^*}(1_s)\delta_{s^*}\cdot \sum_{i=1}^n a_{e_i} \delta_{e_i} \cdot 1_s\delta_s}
                  = \overline{\sum_{i=1}^n\pi_{s^*}(1_s a_{e_i})\delta_{s^*e_i}\cdot 1_s\delta_s} \\
                 & \stackrel{s^*e_i \leq s^*}{=} \overline{\sum_{i=1}^n\pi_{s^*}(1_s a_{e_i})\delta_{s^*}\cdot 1_s\delta_s} 
                  = \overline{\sum_{i=1}^n\pi_{s^*}(1_s a_{e_i} 1_s)\delta_{s^*s}} \\
                 & = \overline{ \pi_{s^*} \left( 1_s \left( \sum_{i=1}^n a_{e_i} \right) 1_s \right) \delta_{s^*s}} 
                  = \overline{\pi_{s^*}(a_s)\delta_{s^*s}}
\end{align*}
is in $\J \cap \A = \mathcal{K}$ and hence $\pi_{s^*}(a_s) = \phi^{-1}(\overline{\pi_{s^*}(a_s)\delta_{s^*s}}) \in \mathcal{K}'$. Therefore $\mathcal{K}'$ is $S$-invariant.
Using that $\A$ is $S$-simple we conclude that $\mathcal{K}'=\A.$  

Now, consider the arbitrary element $\overline{a_s\delta_s} \in \Skew.$
By letting $1_s$ be a local unit for $a_s$ in $D_s,$ we have that $1_s \in \A =\mathcal{K}'$. Hence there are idempotents $f_1, \ldots, f_m \in E(S)$ and $u_j \in D_{f_j}$, for $j\in \{1,\ldots,m\}$, such that $1_s=\sum_{j=1}^m u_j \in \mathcal{K}'$ and  $\phi(1_s)=\sum_{j=1}^m\overline{u_j\delta_{f_j}} \in \mathcal{K} \subseteq \J.$ Thus,
\begin{align*}
\overline{a_s\delta_s}
&=\overline{1_sa_s\delta_s}
=\overline{ \left( \sum_{j=1}^m u_j \right) a_s\delta_{s}}
=\overline{\sum_{j=1}^mu_ja_s\delta_{s}}
\ \stackrel{f_js \leq s}{=} \
 \overline{\sum_{j=1}^mu_ja_s\delta_{f_js}} \\
&=\overline{ \left( \sum_{j=1}^mu_j\delta_{f_j} \right) (a_s\delta_s)}  
=\left( \sum_{j=1}^m\overline{u_j\delta_{f_j}} \right) \overline{a_s\delta_s} \ \in \J.
\end{align*}
This shows that $\Skew =\J$ as desired.
\end{proof}

\section{An application to topological dynamics}\label{topd}

In this section we will apply our main results and connect topological properties of a partial action of an inverse semigroup $S$ on a topological space $X$ with algebraic properties of the associated skew inverse semigroup ring $\Lc\rtimes_\alpha S$.

\begin{definition}\label{def:partialactiontopological}
A \emph{topological partial action} of an inverse semigroup $S$ on a locally compact Hausdorff space $X$ is a collection of open sets  $\{X_s\}_{s\in S}$ of $X$ and homeomorphisms  $\{\theta_s : X_{s^*} \to X_s \}_{s\in S}$ such that, for any $s,t\in S$, the following three assertions hold:
\begin{enumerate}[{\rm (i)}]
	\item $X$ is the  union $\bigcup_{e \in E(S)} X_e;$
	\item $\theta_s(X_{s^*} \cap X_t ) = X_s \cap X_{st}$;
	\item $\theta_{s}(\theta_{t}(x))=\theta_{st}(x)$, for all $x\in X_{t^*} \cap X_{t^*s^*}$.
\end{enumerate}
\end{definition}

Let $X$ be a
zero-dimensional locally compact Hausdorff space (this means that $X$ has a basis formed by compact-open subsets) and let $R$ be a unital commutative ring.  We denote by $\Lc$ the commutative $R$-algebra  formed by  all locally constant, compactly supported, $R$-valued functions on $X$ (with pointwise addition and multiplication). If $D$ is a compact-open subset of $X$, then the characteristic function of $D$, $1_D$, is an element of $\Lc$. Moreover, every $f \in \Lc$ may be written as
\begin{displaymath}
	f=\sum_{i=1}^n c_i1_{D_i},
\end{displaymath}
where the $D_i$'s are compact-open, pairwise disjoint subsets of $X$. For $f \in \Lc$, we define the \emph{support of $f$} by 
\begin{displaymath}
	\supp(f)= \{x \in X \ \mid \ f(x)\neq 0\},
\end{displaymath}
and notice that it is a clopen set. Notice that an $R$-valued function is locally constant if, and only if, it is continuous
once we equip $R$ with the discrete topology. 

For a subset $T \subseteq X,$ we define $I(T)=\{f \in \Lc \ \mid \ f(x) = 0, \ \forall x \in T\}.$ 
Clearly, the set $I(T)$ is an ideal of $\Lc.$ Moreover, since every function in $\Lc$ is continuous, we conclude that $I(T)=I{\left(\overline{T}\right)},$ where $\overline{T}$ denotes the closure of $T.$ 

\begin{lemma}\label{lemmaideal}
Let $R$ be a field. Then every ideal $J$ of $\Lc$ is of the form
\begin{displaymath}
	I(F):= \{ f \in \Lc \, \mid \, f(x) = 0, \, \forall x \in F \},
\end{displaymath}
where $F$ is a closed subset of $X$ given by 
\begin{displaymath}
	F=\{x \in X \, \mid \, f(x)=0, \ \forall f \in J \}.
\end{displaymath}
\end{lemma}

\begin{proof}
Let $J$ be an ideal of $\Lc.$ Using that every function $f \in \Lc$ is continuous, 
we have that the subset $F=\{x \in X \, \mid \, f(x) =0, \ \forall f \in J \}$
is closed in $X$. Clearly,  $J \subseteq I(F). $

Now, take any $f \in I(F).$ Consider the set $U=\supp (f).$ Notice that $U \cap F = \varnothing .$
If $x \in U$, then $x \notin F$ and there exists some $f_x \in J$ such that $f_x(x) \neq 0$. 
We have that
\begin{displaymath}
	U \subseteq \bigcup_{x \in U} \{y \in X \ \mid \ f_x(y)\neq 0 \}= \bigcup_{x \in U} \supp(f_x).
\end{displaymath}
By compactness of $U$ we may find finitely many points $x_1, \ldots, x_n$ such that
\begin{displaymath}
	U \subseteq \bigcup_{i=1}^n \{y \in X \ \mid \ f_{x_i}(y)\neq 0 \}=\bigcup_{i=1}^n \supp(f_{x_i}).
\end{displaymath}

Consider $U_1:=\supp(f_{x_1})$ and $U_i:= \supp(f_{x_i}) \setminus \bigcup_{k=1}^{i-1} \supp(f_{x_k})$ 
for all $i \in \{2,\ldots,n\}.$ We have that 
\begin{displaymath}
	\bigcup_{i=1}^n \supp(f_{x_i})=\bigcup_{i=1}^n U_i,
\end{displaymath}
where the last union is a disjoint union of compact-open subsets. 

Let $g:=\sum_{i=1}^n f_{x_i}\cdot 1_{U_i}.$ 
Using that $f_{x_i} \in J$, for each $i \in \{1,\ldots,n\}$, we have that $g \in J.$ 
Notice that  $g(x) \neq 0$ for all $x \in U.$ 
We define
\begin{displaymath}
	h(x):= \left\lbrace 
\begin{array}{ccl}
\frac{1}{g(x)} & & \mbox{if} \ x \in U \\
0              & & \mbox{if} \ x \notin U
\end{array}
\right.
\end{displaymath}
and notice that $h \in \Lc$. 
Clearly, $f=f\cdot g \cdot h \in J.$
Actually, $g \cdot h$ is a local unit for $f$.
\end{proof}

\begin{remark}\label{idealu}
Let $R$ be a field. Notice that, by Lemma~\ref{lemmaideal}, every ideal $J$ of $\Lc$ is of the form
\begin{displaymath}
	I(U)=\left\lbrace f \in \Lc \, \mid \, f(x)=0, \, \forall x \in X \setminus U \right\rbrace = \left\lbrace f \in \Lc \, \mid \, \supp(f)\subseteq U \right\rbrace,
\end{displaymath}
where $U$ is an open subset of $X$ defined as 
\begin{displaymath}
	U=\left\lbrace x \in X \, \mid \, \exists f \in J \, \mbox{ such that } \, f(x) \neq 0 \right\rbrace=\bigcup_{f \in J} \supp(f).
\end{displaymath}
\end{remark}

From the partial action $\theta=(\{\theta_s \}_{s\in S}, \{X_s\}_{s\in S})$ of an inverse semigroup $S$ on a locally compact, Hausdorff, zero-dimensional space $X,$ we get a corresponding partial action $\alpha=(\{\alpha_s\}_{s\in S}, \{D_s\}_{s\in S})$ of the semigroup $S$ on the $R$-algebra $\Lc$ of all locally constant, compactly supported, $R$-valued functions on $X,$ where $R$ is a unital and commutative ring. More precisely, for each $s \in S,$ we have that $\alpha_s$ is an isomorphism from
\begin{displaymath}
	D_{s^*}=\{f \in \Lc  \ \mid \ \supp(f) \subseteq X_{s^*}\} \simeq \mathcal{L}_c(X_{s^*})
\end{displaymath}
onto 
\begin{displaymath}
	D_s=\{f \in \Lc  \ \mid \ \supp(f) \subseteq X_s\} \simeq \mathcal{L}_c(X_s)
\end{displaymath}
which is defined by
\begin{displaymath}
	\alpha_s(f)(x)= \left\lbrace \begin{array}{ccr}
    f \circ \theta_{s^*}(x) & & \mbox{if} \,\, x \in X_s \\
    0                       & & \mbox{if} \,\, x \notin X_s 
    \end{array}\right..
\end{displaymath}

\begin{remark}\label{pinondegenerate}
It is routine to check that $\alpha$ is a partial action. We wish to convince the reader that $\alpha$ is non-degenerate. Let $f$ be an element of $\Lc$.  By non-degeneracy of $\theta$, for any $x \in \supp(f)$ there is  a compact-open neighbourhood $D$ of  $x$ contained in $X_e$, for some $e \in E(S)$, and such that $f\lvert_D$ is constant. By compactness of $\supp(f)$ we can find finitely many compact-open subsets $D_1, \ldots, D_n$ such that $D_i \subseteq X_{e_i}$, $e_i \in E(S)$, and $\supp(f) \subseteq \bigcup_{i=1}^n D_i$.
By putting $K_1=D_1\cap \supp(f)$ and $K_j=\left(D_j\setminus \bigcup_{i=1}^{j-1}(D_i)\right)\cap \supp(f)$, for all $j \in \{2, \ldots, n\}$, we get that $\supp(f)$ is equal to the disjoint union of the compact-open subsets $K_1,\ldots,K_n$, and that
\[f=\sum_{i=1}^n f 1_{K_i}= \sum_{i=1}^n r_i 1_{K_i} \in \Span_R \left(\bigcup_{e\in E(S)} D_e\right).\]
Notice that the $K_i$'s are pairwise disjoint subsets.
\end{remark}

\begin{definition}
Let $\theta=(\{\theta_s \}_{s\in S}, \{X_s\}_{s\in S})$ be a topological partial action of an inverse semigroup $S$ on a locally compact Hausdorff space $X$. A subset $U$ of $X$ is \emph{invariant} if $\theta_s(U\cap X_{s^*}) \subseteq U$ for all $s \in S$.  The topological partial action $\theta$ is \emph{minimal} if there is no non-empty, proper and open invariant subset of $X$.
\end{definition}

Let $R$ be a field. It is easy to see that if $U$ is an open invariant subset of $X$ then the associated ideal $I(U)$ is invariant. Conversely, every invariant ideal corresponds to an open invariant subset of $X$. Indeed, suppose that $I$ is an invariant ideal of $\Lc$. By Remark~\ref{idealu} there is an open subset $U$ of $X$ such that $I=I(U)$.
Take $s\in S$, $x \in U\cap X_{s^*}$, and suppose that $\theta_s(x) \notin U$.
Let $K \subseteq U$ be a compact-open neighbourhood of $x$ (it exists since $X$ is zero-dimensional). Notice that the function $1_K$ is contained in $I(U)$. Since $I$ is invariant, $\alpha_{s}(1_K)\in I(U)$, that is, $1_K\circ \theta_{s^*} \in I(U)$. But then we get that 
\begin{displaymath}
	1=1_K(x)=1_K (\theta_{s^*}\circ\theta_{s}(x))=1_K \circ \theta_{s^*} (\theta_{s}(x))=0.
\end{displaymath}
Therefore $U$ is invariant.

From the previous paragraph we obtain the following result.

\begin{prop}\label{lem:MinimalSsimple}
Let $R$ be a field and let $\theta=(\{\theta_s \}_{s\in S}, \{X_s\}_{s\in S})$ be a topological partial action of an inverse semigroup $S$ on a zero-dimensional locally compact Hausdorff space $X$. Then $\theta$ is minimal if, and only if, $\Lc$ is $S$-simple (with respect to the action $\alpha$ associated with $\theta$).
\end{prop}

The notion of a topologically free partial action is already well-known for partial group actions. Recall that, if $G$ is a group with neutral element $1$ then a topological partial action $\theta = \left(\{X_t\}_{t\in G},\{\theta_t\}_{t\in G}\right)$ of $G$ on $X$ is
\emph{topologically free} if, for all $t\neq 1$, the set
\[\Lambda_t(\theta) := \{ x \in X_{t^{-1}} \mid \theta_t(x)\neq x \}\]
is dense in $X_{t^{-1}}$.
This is equivalent to saying that, for all $t\neq 1$, the set 
\[F_t(\theta) := \{ x \in X_{t^{-1}}  \mid  \theta_t(x)=x \}\]
has empty interior. 

\begin{prop}\label{prop:freegroup}
Suppose that $\theta=\left(\{X_t\}_{t\in G},\{\theta_t\}_{t\in G}\right)$ is a topologically free partial action of a group $G$ on $X$. Then $\overline{\Lc\delta_1}$ is maximal commutative in $\Lc \rtimes_\alpha G$.
\end{prop}

\begin{proof}
The proof is analogous to the proof of \cite[Proposition 4.7]{GOR}.
\end{proof}

With the intention of generalizing the above result, we will now present and use the definition of topologically principal partial actions of inverse semigroups, which was introduced in \cite{beuterandcordeiro}, and the definition of topologically free (effective)  partial actions of inverse semigroups,  which was introduced in \cite{ExelPardo}.

If $\theta$ is a partial action of an inverse semigroup $S$ on a set $X$, then for any fixed $x\in X$ we define the subset $S_x := \left\{s\in S \mid x\in X_{s^*}\right\}$ of $S$.

\begin{definition}(\cite[Definition~7.1]{beuterandcordeiro})
Let $\theta=\left(\{X_s\}_{s\in S},\{\theta_s\}_{s\in S}\right)$ be a topological partial action of an inverse semigroup $S$ on a locally compact Hausdorff space $X.$ We define
\[\Lambda(\theta):=\left\{x\in X \mid \forall s\in S_x\text{, if }\theta_s(x)=x\text{ then there is }e\in E(S)\text{ with }e\leq s\text{ and }x\in X_e\right\}.\]
We say that $\theta$ is \emph{topologically principal} if, and only if, $\Lambda(\theta)$ is dense in $X$.
\end{definition}

The notion \emph{topologically principal} stems from the fact that, by \cite[Proposition~7.9]{beuterandcordeiro}, the groupoid of germs $S\ltimes X$ associated with a partial action  $\theta=\left(\{X_s\}_{s\in S},\{\theta_s\}_{s\in S}\right)$  is topologically principal if, and only if, the partial action $\theta$ is topologically principal.

\begin{remark}
Notice that if $\theta=\left(\{X_s\}_{s\in S},\{\theta_s\}_{s\in S}\right)$ is a topological partial action of $S$ on $X$ then
\[\Lambda(\theta)=\left\{x\in X \mid \forall s,t\in S_x  \text{ if }\theta_s(x)=\theta_t(x)\text{ then there exists } u\leq s,t\text{ with }x\in X_{u^*} \right\}.\]
In particular, given $x\in X$ and $s,t\in S_x$, if $\theta_s(x)=\theta_t(x)$ and there is $u\leq s,t$ with $x\in X_{u^*}$, then $\theta_s$ and $\theta_t$ coincide in the neighbourhood $X_{u^*}$ of $x$.
\end{remark}

\begin{prop}\label{prop:principalcountable}
Let $S$ be a countable inverse semigroup and let $X$ be a locally compact, Hausdorff, and second-countable topological space $X$. Then a topological partial action $\theta=\left(\{X_s\}_{s\in S},\{\theta_s\}_{s\in S}\right)$ of $S$ on $X$ is topologically principal if, and only if, for any $s\in S$ the set
\begin{displaymath}
\Lambda_{s}(\theta):=\left\{x\in X_{s^*} \mid \text{if }\theta_s(x)=x\text{ then there is }e\in E(S)\text{ with }e\leq s\text{ and }x\in X_e\right\}
\end{displaymath}
is dense in $X_{s^*}$.
\end{prop}

\begin{proof}
The ''only if'' statement is easy to show, using that
\begin{displaymath}
	\Lambda(\theta) \cap X_{s^*} \subseteq \Lambda_{s}(\theta).
\end{displaymath}
Notice that in this direction we do not need to use the fact that $S$ is countable.

Now we show the ''if'' statement.
Notice that $\Lambda_{s}(\theta)$ is an open subset of $X$ and that  $\Lambda_{s}(\theta) \cup \text{int}(X\setminus X_{s^*})$ is dense in $X$.
Thus
\begin{displaymath}
	\Lambda(\theta)=\bigcap_{s\in S}\left( \Lambda_{s}(\theta) \bigcup \text{int}(X\setminus X_{s^*})\right)
\end{displaymath}
is dense in $X$ by the Baire category theorem.
\end{proof}

\begin{definition}\label{def:topfree} (\cite[Definition~4.1]{ExelPardo})
Let $\theta=\left(\{X_s\}_{s\in S},\{\theta_s\}_{s\in S}\right)$ be a topological partial action of an inverse semigroup $S$ on a locally compact Hausdorff space $X.$ We say that $\theta$ is \emph{topologically free} or \emph{effective} if, and only if,
\begin{displaymath}
 \Int \left\lbrace x \in X_{s^*} \mid  \theta_{s}(x)=x \right\rbrace = \{ x \in X_{s^*}  \mid \text{ there is } \ e \in E(S) \text{ such that } e\leq s \text{ and } x \in X_e \}
\end{displaymath}
for all $s \in S$.
\end{definition}

\begin{remark}
(a) If $S$ is a group, then Definition~\ref{def:topfree} coincides with the usual definition of a topologically free partial action of a group.

(b) By \cite[Thereom~4.7]{ExelPardo},
the groupoid of germs $S\ltimes X$ associated with a topological partial action $\theta=\left(\{X_s\}_{s\in S},\{\theta_s\}_{s\in S}\right)$  is effective if, and only if, $\theta$  is topologically free.
\end{remark}

\begin{lemma}\label{lem:topfreeimpliestopprincipal}
Let $S$ be a countable inverse semigroup and let $X$ be a locally compact, Hausdorff, and second-countable topological space $X$.
If $\theta=\left(\{X_s\}_{s\in S},\{\theta_s\}_{s\in S}\right)$ is a topologically free partial action of $S$ on $X$, then $\theta$ is topologically principal.
\end{lemma}

\begin{proof}
Suppose that $\theta$ is not topologically principal. We will show that $\theta$ is not topologically free.
By Proposition~\ref{prop:principalcountable}, there is some $s\in S$ such that $\Lambda_s(\theta)$ is not dense in $X_{s^*}$. 
Now, pick some $y \in X_{s^*}$ such that $y \notin \Lambda_s(\theta)$ and
$y$ is not a limit point of $\Lambda_s(\theta)$.
Notice that $\{x\in X_{s^*} \mid \theta_s(x) \neq x\} \subseteq \Lambda_s(\theta)$.
Clearly,  $y\in \{x\in X_{s^*}\mid \theta_s(x)=x\}$. Moreover, there is an open neighbourhood $U$ of $y$ such that
$U\cap\Lambda_s(\theta)=\emptyset$. Thus, $U \subseteq \{x\in X_{s^*}\mid \theta_s(x)=x\}$. This shows that $y \in \Int \{x\in X_{s^*}\mid \theta_s(x)=x\}$
and therefore $\theta$ is not topologically free.
\end{proof}

The next example shows that the conclusion of  Lemma~\ref{lem:topfreeimpliestopprincipal} does not hold for an arbitrary topological space $X$.

\begin{example}
Let $K$ denote the Cantor set and equip $ \mathbb{T}=\{e^{i\omega} \mid \omega \in \mathbb{R}\}$ with the discrete topology.
Consider the topological product space
$X=(K\cap (0,1))\times \mathbb{T}$. Define an action $\theta$ of the additive group $\mathbb{R}$ on $X$ by
\begin{displaymath}
    \theta_t(s, e^{i\omega})=(s, e^{i(\omega+2st\pi)})
\end{displaymath}
for $t\in \mathbb{R}$ and $(s,e^{i\omega}) \in X$.
For any $t\in\mathbb{R}\setminus \{0\}$, we have that
\[\Int\{(s,e^{i\omega}) \in X \mid \theta_t(s,e^{i\omega})=(s,e^{i\omega})\}=\Int\{(s,e^{i\omega}) \in X \mid st\in \mathbb{Z}\}=\emptyset,\]
and therefore $\theta$ is topologically free.
However, $\theta$ is not topologically principal.
Indeed, let $x=(s,e^{i\omega}) \in X$ be arbitrary.
Put $t=\frac{1}{s}$ and notice that
$\theta_{t}(x)=\theta_{\frac{1}{s}}(s, e^{i\omega})=(s,e^{i(\omega+2\pi)})=(s, e^{i\omega})=x$.
But $0$ is the only idempotent element of the additive group $\mathbb{R}$,
and using that
$t=\frac{1}{s} \neq 0$
we conclude that
$0 \not \leq \frac{1}{s}$. In other words, $(s,e^{i\omega}) \notin \Lambda(\theta)$.
This shows that $\Lambda(\theta)=\emptyset$, and in particular $\theta$ is not topologically principal.
\end{example}

The next example shows that the converse of Lemma~\ref{lem:topfreeimpliestopprincipal} does not hold. That is, there is a topologically principal partial action $\theta$ of a countable inverse semigroup $S$ on a second-countable space $X$ (locally compact, Hausdorff and zero-dimensional) such that $\theta$ is not topologically free.

\begin{example}\label{ex:two-headed-snake}
We shall consider a particular case of the \emph{Munn representation} (see \cite{MR0262402}). Given an inverse semigroup $S$, we consider the set $X=E(S)$ equipped with the discrete topology. Moreover, for $s\in S$, we put $X_s=\left\{e\in E(S) \mid e\leq ss^*\right\}$ and $\theta_s(e)=ses^*$, for all $e\in X_{s^*}$. Then  $\theta=\left(\{X_s\}_{s\in S},\{\theta_s\}_{s\in S}\right)$ is the Munn representation of $S$.

Let us now consider the inverse semigroup $S=\mathbb{N}\cup\{\infty, z\}$ whose product, for any $m,n\in \mathbb{N}$, is given by
\begin{displaymath}
nm=\min(n,m),\quad n\infty=\infty n=nz=zn=n,\qquad z\infty=\infty z=z\quad\text{and}\quad zz=\infty\infty=\infty. 
\end{displaymath}
Then $X=E(S)=\mathbb{N}\cup\left\{\infty\right\}$, can be seen as the one-point compactification of the natural numbers. Notice that the compact-open sets of $X$ are either cofinite or contained in $\mathbb{N}$. Now, let $\theta$ be the Munn representation of $S=\mathbb{N}\cup\{\infty, z\}$. More precisely,  
\begin{itemize}
\item for $n \in \mathbb{N}$, $X_n=\{1, 2, \ldots, n\}$ and $\theta_n=\Identity_{X_n}$; 
\item $X_\infty=\mathbb{N}\cup\{\infty\}$ and $\theta_\infty=\Identity_{X_\infty}$;
\item $X_z=\mathbb{N}\cup\{\infty\}$ and $\theta_z=\Identity_{X_z}$.
\end{itemize}
Notice that $S$ is countable and that $X$ is second-countable.

Since $n, \infty \in E(S)$, clearly $\Lambda_{n}(\theta)=X_n$ and $\Lambda_{\infty}(\theta)=X_\infty$.
Moreover, we notice that $\Lambda_{z}(\theta)= \mathbb{N}$ which is a dense subset of $\mathbb{N}\cup \{\infty\}=X_z$. Hence, by Proposition~\ref{prop:principalcountable},  $\theta$ is topologically principal.

Notice that $\theta$ is not a topologically free partial action. Indeed, 
\[ \Int\{x \in X_{z^*} \mid \theta_z(x)=x\}=\Int\left(\mathbb{N}\cup\{\infty\}\right)=\mathbb{N}\cup\{\infty\}, \]
but $\{x\in X_{z^*} \mid \text{ there is } e \in E(S) \text{ such that } e \leq z \text{ and } x \in X_e \}= \mathbb{N}$.

Let $\alpha$ be the partial action of $S$ on $\Lc$ associated with $\theta$, then 
\begin{itemize}
\item for $n \in \mathbb{N}$, $D_n\simeq \mathcal{L}_c(\{1, 2, \ldots, n\})$ and $\alpha_n=\Identity_{D_n}$. 
\item $D_\infty\simeq \mathcal{L}(\mathbb{N}\cup\{\infty\})$ and $\alpha_\infty=\Identity_{D_\infty}$.
\item $D_z=\mathcal{L}(\mathbb{N}\cup\{\infty\})$ and $\alpha_z=\Identity_{D_z}$.
\end{itemize}

Now we will see that the diagonal of $\Lc\rtimes_{\alpha} S$ is not a maximal commutative subring. Fix $n\in \mathbb{N}$. We denote by $1_{[n, \infty]}$ the characteristic function of the set $\{n, n+1, n+2, \ldots \}\cup\{\infty\}$. We have that $1_{[n, \infty]} \in D_z \simeq \mathcal{L}(\mathbb{N}\cup \{\infty \})$ and that $\overline{1_{[n, \infty]}\delta_z}$ does not belong to the diagonal $\mathcal{D}$ of $\Lc\rtimes_{\alpha} S$.
It is not difficult to see that $\overline{1_{[n, \infty]}\delta_z}$ commutes with all elements of the diagonal.

It is also worth noticing that in this example the ring $\Lc\rtimes_\alpha S$ has an infinite number of non-zero ideals whose intersection with the diagonal is zero. Indeed, for each $n \in \mathbb{N}$, consider the ideal  $\mathcal{J}$ generated by
\[\overline{1_{[n,\infty]}\delta_z} - \overline{1_{[n,\infty]}\delta_\infty}\]
and notice that 
\[\mathcal{J}=\left\lbrace \sum_{i=1}^k \overline{r_i1_{[n_i,\infty]}\delta_z}- \overline{r_i1_{[n_i,\infty]}\delta_\infty} \ \ \Big\lvert \ \
n_i \geq n \text{ and } r_i \in R \right\rbrace.\]
In this case, $\supp\left(\tau{\left(\overline{f}\right)}\right)=\emptyset$,
for every $\overline{f} \in \mathcal{J}$.
\end{example}

Next we present a sufficient condition to obtain the ideal intersection property for the skew inverse semigroup ring arising from a topologically principal partial action. 

\begin{prop}\label{prop:topprincipal}
Let $\theta=(\{\theta_s \}_{s\in S}, \{X_s\}_{s\in S})$ be a topologically principal partial action of $S$ on a zero-dimensional locally compact Hausdorff space $X$.
If $\mathcal{I}$ is a non-zero ideal of $\Lc \rtimes_{\alpha} S$, and there is some $\overline f \in \mathcal{I}$ such that $\supp\left(\widetilde\tau\left(\overline f\right)\right)\neq \emptyset$, then $\mathcal{I}\cap \mathcal{D} \neq \{0\}$.
\end{prop}

\begin{proof}
Let $\mathcal{I}$ be a non-zero ideal of $\Lc\rtimes_{\alpha} S$ and let    \[\overline f = \sum_{s \in F} \overline{f_s\delta_s} \in \mathcal{I}\] be such that $\supp(\widetilde\tau(\overline f))\neq \emptyset$. Since $\supp(\widetilde\tau(\overline f))$ is an open subset of $X$, and $\theta$ is  topologically principal, there is some $x \in \supp\left(\widetilde\tau\left(\overline f\right)\right)\cap \Lambda(\theta)$. We fix this $x$ throughout the rest of the proof.

Notice that the subset
\[\left\lbrace \theta_{s^*}(x) \mid s\in F \text{ and } f_s(x)\neq 0 \right\rbrace \]
is non-empty and finite.

We choose $s_1\in F$ such that $f_{s_1}(x)\neq 0$ and 
\[r:=\displaystyle\sum_{s \in T} f_s(x) \neq 0,\]
where $T:=\{s\in F \mid f_s(x)\neq 0 \text{ and } \theta_{s^*}(x)=\theta_{s_1^*}(x)\}$.

Let $y:=\theta_{s_1^*}(x)$. Furthermore, let $B$ be a compact-open neighbourhood of $x$ contained in $X_{s_1}\cap X_e$, for some $e \in E(S)$ such that 
\[\left\lbrace \theta_{s^*}(x) \ \mid \ s\in F \text{ and } f_s(x) \neq 0 \right\rbrace \cap B = \left\lbrace y \right\rbrace.\]
We have that $\overline{g}=\overline{f} \cdot \overline{1_B\delta_e} \in \mathcal{I}$ and that
\[\overline{g}=\sum_{s\in F} \overline{\alpha_s(\alpha_{s^*}(f_s)1_B)\delta_{es}}=  \sum_{s\in F} \overline{\alpha_s(\alpha_{s^*}(f_s)1_B)\delta_s}.\]
Put  $g_s:=\alpha_s(\alpha_{s^*}(f_s)1_B)$, and notice that $\supp(g_s)\subseteq \supp(f_s)\cap \theta_s(B)$. Then
\[\left\lbrace\theta_{s^*}(x) \mid s \in F \text{ and } g_s(x)\neq 0 \right\rbrace=\{y\}.\]
Furthermore,
\[\{s\in F \mid g_s(x)\neq 0 \}=\{s\in F \mid f_s(x)\neq 0 \text{ and } \theta_{s^*}(x)=\theta_{s_1^*}(x)\}=T.\]

Using that $x \in \Lambda(\theta)$, there is some $u\in S$ such that $x\in X_{u^*}$ and $u\leq s^*$, for all $s \in T$. Let $C\subseteq X_{u^*}$ be a compact-open neighbourhood of $x$. We may now rewrite $\overline{g}$ as
\begin{align*}
\overline{g} 
  & = \sum_{s\in F : g_s(x)=0} \overline{g_s\delta_s} + \sum_{s\in F :g_s(x)\neq 0} \overline{(g_s1_C-g_s1_{X\setminus C})\delta_s} \\
  & = \sum_{s\in F : g_s(x)=0} \overline{g_s\delta_s} + \sum_{s\in F :g_s(x)\neq 0} \overline{g_s1_C\delta_s}-\sum_{s\in F :g_s(x)\neq 0}\overline{g_s1_{X\setminus C}\delta_s} \\
  & = \sum_{s\in F : g_s(x)=0} \overline{g_s\delta_s} + \sum_{s\in F :g_s(x)\neq 0} \overline{g_s1_C\delta_{u^*}}-\sum_{s\in F :g_s(x)\neq 0}\overline{g_s1_{X\setminus C}\delta_s}.
\end{align*}

Since each $g_s$ is locally constant,  we can find another compact-open neighbourhood $K$ of $x$ contained in $C$ such that $g_s|_K$ is constant, for all $s\in F$, and that $K\subseteq X_v$, for some $v\in E(S)$.

Thus, $\overline{1_K\delta_v}\cdot \overline{g}\cdot \overline{\alpha_{u}(1_K)\delta_{u}} \in \mathcal{I}$ and

\begin{align*}
\overline{1_K\delta_v}\cdot \overline{g} \cdot \overline{\alpha_{u}(1_K)\delta_u}
 & = \sum_{s\in F :g_s(x)\neq 0} \overline{1_Kg_s1_C\delta_{vu^*}}\cdot \overline{\alpha_{u}(1_K)\delta_u} \\
 & = \overline{\left(\sum_{s\in F :g_s(x)\neq 0} g_s(x)\right)1_K \delta_{vu^*}}\cdot \overline{\alpha_{u}(1_K)\delta_u}\\
 & = \overline{r1_K \delta_{u^*}}\cdot \overline{\alpha_{u}(1_K)\delta_u}\\
 & = \overline{\alpha_{u^*}(\alpha_u(r1_K)\alpha_u(1_K))\delta_{u^*u}} \\
 &= \overline{r1_K\delta_{u^*u}} \in \mathcal{I}\cap \mathcal{D},
\end{align*}
where $r\neq 0$.
\end{proof}

\begin{corollary}\label{cor:topprincipalmaxcom}
Let $\theta=(\{\theta_s \}_{s\in S}, \{X_s\}_{s\in S})$ be a topologically principal partial action of $S$ on 
a zero-dimensional locally compact Hausdorff space $X$. If for each non-zero ideal $\mathcal{I}$ of $\Lc\rtimes_{\alpha} S$ there is some $\overline{f} \in \mathcal{I}$ such that $\supp(\widetilde\tau(\overline f))\neq \emptyset$, then the diagonal $\mathcal{D}$ is a maximal commutative subring of $\Lc\rtimes_{\alpha} S$.
\end{corollary}

\begin{proof}
This follows from Proposition~\ref{prop:topprincipal} and Theorem~\ref{teo1}.
\end{proof}

Finally, we show that maximal commutativity of $\mathcal{D}$ in $\Lc\rtimes_{\alpha} S$ implies that the underlying action is topologically principal (we also show the condition involving ideals and support of elements in $\Lc\rtimes_{\alpha} S$). 

\begin{prop}\label{lem:maxcommtopprincipal}
Suppose that $S$ is countable, $X$ is a second-countable, zero-dimensional and locally compact Hausdorff space, and $\theta=(\{\theta_s \}_{s\in S}, \{X_s\}_{s\in S})$ is a topological partial action of $S$ on $X$. If the diagonal $\mathcal{D}$ is a maximal commutative subring of $\Lc\rtimes_{\alpha} S$, then $\theta$ is topologically principal, and for each non-zero ideal $\mathcal{I}$ of $\Lc\rtimes_{\alpha} S$ there is some $\overline{f} \in \mathcal{I}$ such that $\supp(\widetilde\tau(\overline f))\neq \emptyset$.
\end{prop}

\begin{proof}
We show the contrapositive statement. Suppose that $\theta$ is not topologically principal. Then there is $s\in S$ and there is a compact-open subset $B$ of $X_{s^*}$ such that $B\cap \Lambda_s(\theta)=\emptyset$.  This means that, for each $x\in B$,  $\theta_s(x)=x$  and there is no $e\in E(S)$ such that $e\leq s$ and $x\in X_e$ (equivalently, there is no $e\in E(S)$ such that $e\leq s^*$ and $x\in X_e$). Hence,  $\overline{1_B\delta_{s^*}} \notin \mathcal{D}.$ 

Let $e\in E(S)$ be arbitrary and take an arbitrary compact-open subset $D$ of $X_e$.
If $x\in B$, then
\[\alpha_{s^*}(\alpha_{s}(1_B)1_D)(x)= 1_B(x)1_D(\theta_{s}(x))=1_B(x)1_D(x)=(1_B1_D)(x).\]
And if $x\in X \setminus B$, then
\[\alpha_{s^*}(\alpha_{s}(1_B)1_D)(x)= 1_B(x)1_D(\theta_{s}(x))=0 =1_B(x)1_D(x)=(1_B1_D)(x).\]
Hence, $\alpha_{s^*}(\alpha_{s}(1_B)1_D)=1_B1_D$. Thus,
\[ \overline{1_D\delta_e} \cdot \overline{1_B\delta_{s^*}} 
 = \overline{1_D1_B\delta_{s^*}} = \overline{1_B1_D\delta_{s^*}}
 = \overline{\alpha_{s^*}(\alpha_{s}(1_B)1_D)\delta_{s^*}}
 = \overline{1_B\delta_{s^*}}\cdot \overline{1_D\delta_e}. \]

This implies that $\overline{1_B\delta_{s^*}}$ commutes with all elements of the diagonal $\mathcal{D}$, and hence $\mathcal{D} $ is not maximal commutative.

By Theorem~\ref{teo1}, for every non-zero ideal $\mathcal{I}$ of $\Lc\rtimes_{\alpha} S$ we have that $\mathcal{I}\cap\mathcal{D}\neq \{0\}$. Let $\overline{f}=\sum_{i=1}^nf_i\delta_{e_i} \in\mathcal{I}\cap\mathcal{D}$ be non-zero. By the isomorphism of the diagonal $\mathcal{D}$ with $\Lc$ we have that $\overline{f}=\overline{0}$ if, and only if, $\sum_{i=1}^nf_{e_i}=0$, and thus $\supp\left(\widetilde\tau\left(\overline{f}\right)\right)=\supp\left(\sum_{i=1}^nf_{e_i}\right)\neq \emptyset$.
\end{proof}


\begin{corollary}\label{cor:simples}
Let $S$ be a countable inverse semigroup, let $X$ be a second-countable, zero-dimensional and locally compact Hausdorff space, and let $R$ be a field.
Then the skew inverse semigroup ring $\Lc \rtimes_{\alpha} S$  is simple  if, and only if, the following three conditions are satisfied

\begin{itemize}
    \item $\theta$ is minimal,
    \item $\theta$ is topologically principal, and 
    \item  for every non-zero ideal $\mathcal{I}$ of $\Lc \rtimes_{\alpha} S$ there is some $\overline{f} \in \mathcal{I}$ such that $\supp(\overline{f})\neq \emptyset$.
\end{itemize}  
\end{corollary}

\begin{proof}
This follows from Proposition~\ref{lem:MinimalSsimple}, Corollary~\ref{cor:topprincipalmaxcom}, Lemma~\ref{lem:maxcommtopprincipal} and  Theorem~\ref{thm:main1}. Notice that for the "if" statement we do not need to use the fact that $S$ is countable and that $X$ is second-countable.
\end{proof}

\begin{corollary}
Let $S$, $X$ and $R$ be as in Corollary \ref{cor:simples}.
Suppose that $\theta=(\{\theta_s \}_{s\in S}, \{X_s\}_{s\in S})$ is a partial action such that the following three assertions hold: 
\begin{itemize}
    \item $\theta$ is topologically free 
    \item $\theta$ is minimal, and
    \item for every non-zero ideal $\mathcal{I}$ of $\Lc \rtimes_{\alpha} S$ there is $\overline{f} \in \mathcal{I}$ such that $\supp(\overline{f})\neq \emptyset.$
\end{itemize}
Then the skew inverse semigroup ring $\Lc \rtimes_{\alpha} S$  is simple.
\end{corollary}

\begin{proof}
This follows from Lemma~\ref{lem:topfreeimpliestopprincipal} and Corrollary~\ref{cor:simples}.
\end{proof}


\section{An application to Steinberg algebras}\label{Sec:Steinberg}

In this section we will apply our main result and obtain a new proof of the simplicity criterion for a Steinberg algebra $A_R(\G)$ associated with a Hausdorff and ample groupoid $\G.$ 

\subsection{Steinberg algebras}

Given a groupoid $\G,$ we denote its unit space by $\Go,$ and its source and range maps 
by $s$ and $r,$ respectively. A \emph{bisection} in $\G$ is a subset $B \subseteq \G$
such that the restrictions of $r$ and $s$ to $B$ are both injective.
A topological groupoid $\G$ is said to be \emph{\'{e}tale} if $\Go$ is locally compact and
Hausdorff, and its source map is a local homeomorphism from $\G$ to $\Go.$
An \'{e}tale groupoid $\G$ is said to be \emph{ample} if $\G$ has a basis of compact bisections.
One can show that a Hausdorff \'{e}tale groupoid is ample if, and only if, $\Go$ is totally disconnected.
In this article, we only consider groupoids which are both Hausdorff and ample. 

A subset $U$ of the unit space $\Go$ of $\G$ is \emph{invariant} if $s(b) \in U$ implies $r(b) \in U.$
We call $\G$ \emph{minimal} if $\Go$ has no nontrivial open invariant subset. 
We let $\Iso(\G)$ denote the \emph{isotropy subgroupoid} of $\G,$ that is, 
$ \Iso(\G) := \{ b \in \G \, \mid \, r(b)=s(b) \}.$
A Hausdorff and ample groupoid $\G$ is said to be \emph{effective} if the interior of $\Iso(\G)$ is $\Go,$ 
or equivalently, for every non-empty compact bisection $B \subseteq \G \setminus \Go,$ there exists some $b \in B$ such that $s(b)\neq r(b).$

Let $R$ be a commutative ring with identity and let $\G$ be a Hausdorff and ample groupoid. 
The \emph{Steinberg algebra} $A_R(\G)$ is the collection of compactly supported locally
constant functions from $\G$ to $R$ with pointwise addition, and
convolution product $(f * g)(b) = \sum_{r(c)=r(b)} f(c)g(c^{-1}b) = \sum_{cd=b}f(c)g(d).$
The support of $f \in A_R(\G),$ is denoted by $\supp(f)=\{ b \in \G \mid f(b)\neq 0\}$ (a clopen subset of $\G$).

By \cite[Proposition~4.3]{steinberg}, every element of $A_R(\G)$ is a linear combination of characteristic functions of pairwise disjoint compact bisections. 
Moreover, by \cite[Proposition~3.12.]{steinberg}, $A_R(\G)$ is a unital ring if, and only if, $\Go$ is compact.

Let $\G$ be a Hausdorff and ample groupoid. The set $\Ga$ of all compact bisections in $\G$ 
is an inverse semigroup under the operations defined by 
$ BC = \{bc \in \G \, \mid \,  b \in B, c \in C \ \mbox{and} \ s(b)=r(c) \},$
and $ B^{-1}= \{ b^{-1} \, \mid \, b \in B \}.$ 
The inverse semigroup partial order  in $\Ga$ is the inclusion of sets. Notice that $E(\Ga)=\{U \in \Ga \ \mid \  U \subseteq \Go \}.$

Given any compact bisection $B$ we define the map $\theta_B: s(B) \rightarrow r(B)$  by 
$\theta_B(u)= r_B(s_B^{-1}(u)).$
The correspondence $B \mapsto \theta_B$ gives an action of $\Ga$ on the unit space $\Go$.

From the  action $\theta$ of the semigroup $\Ga$ on the locally compact, Hausdorff, totally disconnected space $\Go,$ we get a corresponding  action $\alpha$ of the semigroup $\Ga$ on the $R$-algebra $\Lco$ of all locally constant, compactly supported, $R$-valued functions on $\Go,$ where $R$ is a unital and commutative ring. More precisely, for each $B \in \Ga,$ we have that $\alpha_B$ is an isomorphism from  $D_{B^*}=\{f \in \Lco  \ \mid \ \supp(f)\subseteq s(B)\} \simeq \mathcal{L}_c(s(B))$ onto $D_B=\{f \in \Lco  \ \mid \ \supp(f) \subseteq r(B)\} \simeq \mathcal{L}_c(r(B))$ which is defined by
\begin{displaymath}
	\alpha_B(f)(x)= \left\lbrace \begin{array}{ccr}
    f \circ \theta_{B^*}(x) & & \mbox{if} \,\, x \in r(B) \\
    0                        & & \mbox{if} \,\, x \notin r(B)  
    \end{array}\right..
\end{displaymath}

In \cite[Theorem~5.2]{beuter2} Beuter and Gon\c{c}alves showed that any Steinberg algebra can be seen as a skew inverse semigroup ring as follows.

\begin{theorem}\label{steinbergasisgsr}
Let $\G$ be a Hausdorff and ample groupoid, let $\theta$ be the action of the inverse semigroup $\Ga$ on the unit space $\Go,$ and let $\alpha$ be the corresponding  action of $\Ga$ on $\Lco.$ Then the Steinberg algebra $A_R(\G)$ is isomorphic, as an $R$-algebra, to the skew inverse semigroup ring $\Lgo$.
\end{theorem}

\begin{remark}\label{remarkvarphi}
The isomorphism of Theorem~\ref{steinbergasisgsr} is given by the map $\tilde{\psi}:\Lgo \rightarrow A_R(\G)$, which is defined on elements of the form $\overline{f_B\delta_B}$, by
\begin{displaymath}
\tilde\psi(\overline{f_B\delta_B})(x) = \left\lbrace 
\begin{array}{cl}
    f_B(r(x)) &  \mbox{if} \,\, x \in B \\
            0 &  \mbox{if} \,\, x \notin B, 
\end{array}\right.
\end{displaymath}
and extended linearly to $\Lgo.$ 
In the proof of Theorem~\ref{steinbergasisgsr} it was shown that $\tilde{\psi}$ admits a left inverse, namely the map $\varphi : A_R(\G) \rightarrow \Lgo$ defined as follows: Given $f = \sum_{j=1}^n b_j 1_{B_j} \in A_R(\G),$ where the $B_j$'s are pairwise disjoint compact bisections of $\G,$ let
\begin{displaymath}
	\varphi(f) = \varphi\left(\sum_{i=1}^n b_j 1_{B_j}\right) :=\sum_{j=1}^n \overline{b_j 1_{r(B_j)}\delta_{B_j}}.
\end{displaymath}
Actually $\varphi$ is the inverse of $\tilde{\psi},$ and, in particular, it is bijective. 
By the surjectivity of $\varphi,$ given any $f \in \Lgo$ we can write 
\begin{displaymath}
	f= \sum_{j=1}^n  \overline{ b_j 1_{r(B_j)} \delta_{B_j} },
\end{displaymath}
where the $B_j$'s are pairwise disjoint compact bisections of $\G.$
Furthermore, by the injectivity of $\varphi$, if
\begin{displaymath}
	\sum_{j=1}^n  \overline{ b_j 1_{r(B_j)} \delta_{B_j} }= \sum_{j=1}^n  \overline{ c_j1_{r(C_j)} \delta_{C_j} } ,
\end{displaymath}
where the $B_j$'s and $C_j$'s are pairwise disjoint compact bisections, then
\begin{displaymath}
	\sum_{i=1}^n   b_i 1_{B_i} = \sum_{i=1}^n  c_i 1_{C_i}.
\end{displaymath}
\end{remark}

\subsection{Simplicity of Steinberg algebras}

Using that there is a description of Steinberg algebras via skew inverse semigroup rings (which satisfies the assumptions of the previous section), we can apply the results of the previous section to characterize simplicity of Steinberg algebras.
We then obtain a new proof of the following result, which was first proved in \cite{brownetal} for functions over the complex numbers.

\begin{theorem}\cite[Corollary~4.6.]{clarkediemichell}\label{thm:SteinbergSimplicity} Let $\G$  be a Hausdorff and ample groupoid, and let $R$ be a unital and commutative ring. Then the Steinberg algebra $A_R(\G)$ is simple if, and only if, $\G$ is effective, minimal, and $R$ is a field.
\end{theorem}

Let $\mathcal{G}$ be a Hausdorff and ample groupoid, and let $R$ be a unital and commutative ring. Then $A_R(\G) \cong \Lgo$. Our first step towards a proof of the above theorem is to characterize minimality of $\G$ in terms of $\Ga$-simplicity of $\Lco$. We set up notation and prove an auxiliary result below.

We can now prove the following.

\begin{prop}\label{Sec4:PropMinimal}
Let $\G$ be a Hausdorff and ample groupoid, and let $R$ be a field. Then $\G$ is minimal if, and only if, $\Lco$ is $\Ga$-simple.
\end{prop}

\begin{proof}
Suppose that $\G$ is minimal. Let $\J$ be a $\Ga$-invariant non-zero ideal of $\Lco.$ We know that 
\begin{displaymath}
	\J = \{ f \in \Lco \, \mid \, f(x)=0, \, \forall x \in \Go \setminus U \},
\end{displaymath}
where $U$ is an open subset of $\Go$ given by 
\begin{displaymath}
	U=\{u \in \Go \, \mid \, \exists f \in \J \, \mbox{ such that } \, f(u) \neq 0 \}.
\end{displaymath}
 
Notice that, since $\G$ is minimal, if we prove that $U$ is an invariant subset of $\Go$, then $U= \Go$ and hence $\J = \Lco$. We prove the invariance of $U$ below.

Let $x \in \G$ be such that $s(x) \in U.$ Then there exists a function $g \in \J$ such that $g(s(x)) \neq 0.$ Furthermore, we can take $x \in B$, where $B$ is a compact bisection of $\G.$ Since $U$ and $s(B)$ are open, we can consider
\begin{displaymath}
	g \in \J \cap D_{B^*} = \{f \in \Lco \, \mid \, \supp(f) \subseteq U \cap s(B) \}.
\end{displaymath}

Using that $\J$ is $\Ga$-invariant we get that $\alpha_B(g) \in \J$.
Notice that 
\begin{displaymath}
\alpha_B(g)(r(x)) = g(\theta_{B^*}(r(x)))
                  = g(s(r_B^{-1}(r(x))))
                  = g(s(x))
                  \neq 0.	
\end{displaymath}
Therefore, $r(x) \in U$ and hence $U$ is $\Ga$-invariant, as desired.

Now, suppose that $\Lco$ is $\Ga$-simple. Let $U \subseteq \Go$ be a non-empty invariant open subset. Consider the set
\begin{displaymath}
	\J = \{f \in \Lco \, \mid \, f(x)=0, \,\, \forall x \in \Go \setminus U \}.
\end{displaymath}

Clearly, $\J$ is an ideal of $\Lco.$ To see that $\J$ is $\Ga$-invariant, suppose that $B \in \Ga,$ $g \in \J \cap D_{B^*}$, and $x \in \Go \setminus U$. If $x \in r(B)$, then there exists some $y \in B$ such that $x=r(y)$, and hence 
\begin{displaymath}
\alpha_B(g)(x) = \alpha_B(g)(r(y))
               = g(s(r_B^{-1}(r(y))))
               = g(s(y)).
\end{displaymath}
Since $U$ is invariant, and $r(y)=x \notin U,$ we have that $s(y) \notin U$. Hence, $ g(s(y))= 0.$ If $x \notin r(B)$, then from the definition of $\alpha_B,$ we also have that $\alpha_B(g)(x)=0.$ Therefore, $\alpha_B(g) \in \J$, and hence $\J$ is $\Ga$-invariant. Using that $\Lco$ is $\Ga$-simple it follows that $\J=\Lco$ and $U=\Go.$ 
\end{proof}

\begin{prop}\label{Sec4:PropEffective}
Let $\G$ be a Hausdorff and ample groupoid.
Then $\G$ is effective if, and only if, 
\begin{displaymath}
	\mathcal{D} =\left\lbrace \sum_{i=1}^n\overline{f_i\delta_{U_i}} \ \mid \ n  \in \mathbb{N}, U_i \in E(\Ga) \ \mbox{and} \ f_i \in \mathcal{L}_c(r(U_i)) = \mathcal{L}_c(U_i) \right\rbrace \simeq \Lco
\end{displaymath}
is a maximal commutative subring of $\Lgo.$
\end{prop}

\begin{proof}
Suppose that $\G$ is effective. 
We already know that $\mathcal{D}$ is a commutative subring.

Let $0 \neq f = \sum_{i=1}^n  \overline{ r_i 1_{r(B_i)} \delta_{B_i} } \in \Lgo,$ 
where $r_i \in R\setminus \{0\}$ and the $B_i$'s are pairwise disjoint compact bisections 
of $\G$ for all $i \in \{1, \ldots, n\}.$ Suppose that $f$ is an element which commutes
with all elements of $\mathcal{D} .$ We need to show that $f\in \mathcal{D} $.

By the effectiveness of $\G$ it suffices to show that $B_i \subseteq \Iso(\G)$ for every $i \in \{1, \ldots, n\}$ 
(since $B_i$ is open and $B_i \subseteq \text{int} (\Iso(\G))=\Go$).  

To this end, suppose that there exists some $k \in \{1, \ldots, n\}$, and $b \in B_k$, such that $r(b) \neq s(b).$ 
Since $\G$ is Hausdorff, there exists a compact bisection $U \subseteq \Go$ such that 
$r(b) \in U$ and $ s(b) \notin U.$ Notice that $U \in E(\Ga).$

Using that $f$ belongs to the centralizer of $\mathcal{D}$ we have that 
\begin{displaymath}
	\overline{1_U\delta_U} \cdot f = f \cdot \overline{1_U\delta_U}.
\end{displaymath}
This implies that 
\begin{displaymath}
\sum_{i=1}^n \overline{r_i 1_U 1_{r(B_i)}\delta_{UB_i}}  = \sum_{i=1}^n \overline{r_i \alpha_{B_i} (\alpha_{B_i^*}(1_{r(B_i)})1_U)\delta_{B_iU}}.	
\end{displaymath}

Since $UB_i, B_iU \subseteq B_i,$ for all $i \in \{1, \ldots, n \},$ we get that
\begin{equation}\label{eqcommutative}
\sum_{i=1}^n \overline{r_i 1_U 1_{r(B_i)}\delta_{B_i}}  = \sum_{i=1}^n \overline{r_i \alpha_{B_i} (\alpha_{B_i^*}(1_{r(B_i)})1_U)\delta_{B_i}}.
\end{equation}

Developing the left side of \eqref{eqcommutative} we obtain
\begin{displaymath}
	\sum_{i=1}^n\overline{r_i 1_U 1_{r(B_i)}\delta_{B_i}}=\sum_{i=1}^n\overline{r_i 1_{U \cap r(B_i)}\delta_{B_i}}.
\end{displaymath}
For each $i \in \{1,\ldots,n\}$, define $C_i:=r_{B_i}^{-1}(U \cap r(B_i)).$ Notice that $C_i \subseteq B_i$ and $r(C_i)=U \cap r(B_i).$
Thus 
\begin{equation}\label{eqleftside}
\sum_{i=1}^n\overline{r_i 1_U 1_{r(B_i)}\delta_{B_i}} = \sum_{i=1}^n\overline{r_i 1_{U\cap r(B_i)}\delta_{B_i}}=  \sum_{i=1}^n\overline{r_i 1_{r(C_i)}\delta_{C_i}}.                                   
\end{equation}

Now, developing the right side of \eqref{eqcommutative} we get that
\begin{align*}
\sum_{i=1}^n \overline{r_i \alpha_{B_i} (\alpha_{B_i^*}(1_{r(B_i)})1_U)\delta_{B_i}} 
       & = \sum_{i=1}^n \overline{r_i \alpha_{B_i} (1_{s(B_i)}1_U)\delta_{B_i}}
        = \sum_{i=1}^n \overline{r_i \alpha_{B_i} (1_{s(B_i)\cap U})\delta_{B_i}} \\
       & = \sum_{i=1}^n \overline{r_i 1_{r(B_i)\cap \theta_{B_i}(s(B_i) \cap U)} \delta_{B_i}}.
\end{align*}
 
Define $D_i:=r_{B_i}^{-1} (r(B_i)\cap \theta_{B_i}(s(B_i) \cap U)).$ Notice that $D_i \subseteq B_i$ and $r(D_i)=r(B_i)\cap \theta_{B_i}(s(B_i) \cap U).$
Then 
\begin{equation}\label{eqrightside}
\sum_{i=1}^n \overline{r_i \alpha_{B_i} (\alpha_{B_i^*}(1_{r(B_i)})1_U)\delta_{B_i}} = \sum_{i=1}^n \overline{r_i 1_{r(D_i)} \delta_{D_i}}.
\end{equation}

By substituting \eqref{eqleftside} and \eqref{eqrightside} into Equation~\eqref{eqcommutative} we obtain that
\begin{displaymath}
	\sum_{i=1}^n\overline{r_i 1_{r(C_i)}\delta_{C_i}} = \sum_{i=1}^n \overline{r_i 1_{r(D_i)} \delta_{D_i}}.
\end{displaymath}

Since $C_i \subseteq B_i$, for each $i\in \{1,\ldots,n\}$, we have that the $C_i$'s are pairwise disjoint compact bisections, and similarly the $D_i$'s are also pairwise disjoint compact bisections. By Remark~\ref{remarkvarphi} we have that
\begin{displaymath}
	\sum_{i=1}^n b_i 1_{C_i} = \sum_{i=1}^n b_i 1_{D_i} .
\end{displaymath}

Next we evaluate the above equality on the element $b$ of $B_k$ such that $r(b)\neq s(b)$. Since the $B_i$'s are pairwise disjoint we have that $b \notin C_i, b \notin D_i$ for $i \neq k$ and hence
\begin{equation}\label{eqeffective}
b_k 1_{C_k}(b) =  b_k 1_{D_k}(b).
\end{equation}
Notice that
\begin{displaymath}
	b \in C_k=r_{B_k}^{-1}(U \cap r(B_k)) \Longleftrightarrow r(b) \in U \cap r(B_k),
\end{displaymath}
and
\begin{align*}
b \in D_k & \Longleftrightarrow b \in  r_{B_k}^{-1} (r(B_k)\cap \theta_{B_k}(s(B_k) \cap U)) \\
          & \Longleftrightarrow r(b) \in r(B_k)\cap \theta_{B_k}(s(B_k) \cap U) \\
          & \Longleftrightarrow r(b) \in r(B_k)\ \mbox{and} \ r(b) \in r(s_{B_k}^{-1}(s(B_k) \cap U)) \\
          & \stackrel{b \in B_k}{\Longleftrightarrow} r(b) \in r(B_k) \ \mbox{and} \ b \in s_{B_k}^{-1}(s(B_k) \cap U) \\
          & \Longleftrightarrow r(b) \in r(B_k) \ \mbox{and} \ s(b) \in s(B_k) \cap U \\
          & \stackrel{b \in B_k}{\Longleftrightarrow} r(b) \in r(B_k) \ \mbox{and} \ s(b) \in U.
\end{align*}

Recall that, by construction, $b\in C_k$ and $s(b)\notin U$. Thus, Equation~\eqref{eqeffective} yields $b_k =0,$ a contradiction.
Therefore, $r(b)=s(b),$ $b \in \Iso(G)$ and $B_i \in  \Iso(G)$ as desired. \\

In order to prove the converse we show the contrapositive statement. Suppose that $\G$ is not effective. Then
there exists a bisection $B \subseteq \G \setminus \Go$ such that  $s(b)=r(b)$ for all $b \in B.$ 

Recall that $\theta_B: s(B)\rightarrow r(B)$ is defined by $r(u)=r(s_B^{-1}(u)).$ Thus, in this case, $\theta_B(s(b))=r(s_B^{-1}(s(b)))=r(b)=s(b),$ that is, $\theta_B=\id_{s(B)}.$ Similarly, $\theta_{B^*}=\id_{r(B)}.$ This implies that $\alpha_B = \id_{D_{B^*}}$ and $\alpha_{B^*} = \id_{D_B}.$  

Notice that $\overline{1_{r(B)} \delta_B} \notin \mathcal{D} .$ 
Take any $\overline{f\delta_U} \in \mathcal{D} .$ Then 
\begin{align*}
\overline{f\delta_U} \cdot \overline{1_{r(B)} \delta_B} \quad & = \quad\quad \overline{f 1_{r(B)} \delta_{UB} }
                                                        \quad\quad \stackrel{\mathclap{UB \subseteq B}}{=} \quad\quad \overline{f 1_{r(B)} \delta_{B} } \quad
                                                        = \quad \overline{1_{r(B)}  f \delta_B} \\
                                                        & \stackrel{\mathclap{r(B)=r(UB)}}{=} \quad\quad \overline{1_{r(B)}  f \delta_{BU}} 
                                                        \quad = \quad \overline{\alpha_B(\alpha_{B^*}(1_{r(B)} f )) \delta_{BU}}
                                                        \quad = \quad \overline{1_{r(B)} \delta_B} \cdot  \overline{f\delta_U},
\end{align*}
that is, $\overline{1_{r(B)} \delta_B}$ commutes with all of $\mathcal{D} $. This shows that $\mathcal{D} $ is not maximal commutative.
\end{proof}

\begin{remark} Since $\mathcal{D} $ is isomorphic to $\mathcal{L}_c(\Go) \simeq A_R(\Go)$ it follows from Proposition~\ref{Sec4:PropEffective}, and Theorem~\ref{teo1}, that $\G$ is effective if, and only if, $A_R(\G_0)$ is maximal commutative if, and only if, every non-zero ideal $I$ of $A_R(\G)$ has non-zero intersection with $A_R(\Go)$. The characterization of effectiveness in terms of the ideal intersection property was first given in \cite{Clark2} and that effectiveness of $\G$ implies maximal commutativeness of $A_R(\G_0)$ was first proven in \cite{Ben}.
\end{remark}

In order for us to apply Theorem~\ref{thm:main1} we need to verify that the assumption about the local units is satisfied. In fact, for any finite subset $\{f_1, \cdots, f_n\}$ of $\mathcal{L}_c(\Go)$ consider $U=\bigcup_{i=1}^n \supp(f_i)$. Clearly,  $1_U \in \mathcal{L}_c(\Go)$ and this element is local unit for $\{f_1, \cdots, f_n\}$. Moreover, $1_{r(B)}$ and $1_{s(B)}$ are multiplicative identity elements in $D_B$ and $D_{B^*},$ respectively.

\begin{remark}
The proof of Theorem~\ref{thm:SteinbergSimplicity}
follows from
Theorem~\ref{thm:main1},
Proposition~\ref{Sec4:PropMinimal}
and Proposition~\ref{Sec4:PropEffective}.
\end{remark}


\begin{thebibliography}{99}

\bibitem{arandapinoetal}
G. Aranda Pino, J. Clark, A. an Huef and I. Raeburn,
{\it Kumjian-Pask algebras of higher-rank graphs},
Trans. Amer. Math. Soc. \textbf{365}(7) (2013), 3613--3641.

\bibitem{beuterandcordeiro}
V. Beuter and L. Cordeiro, 
{\it The dynamics of partial inverse semigroup actions},
arXiv:1804.00396 [math.RA]

\bibitem{beuter1}
V. M. Beuter and D. Gon\c{c}alves,
{\it Partial crossed products as equivalence relation algebras},
Rocky Mountain J. Math. \textbf{46}(1) (2016), 85--104.

\bibitem{beuter2}
V. M. Beuter and D. Gon\c{c}alves,
{\it The interplay between Steinberg algebras and partial skew rings},
arXiv:1706.00127 [math.RA]

\bibitem{boavaexel}
G. Boava and R. Exel,
{\it Partial crossed product description of the C*-algebras associated with integral domains},
Proc. Amer. Math. Soc. \textbf{141}(7) (2013), 2439--2451.

\bibitem{brownetal}
J. Brown, L. O. Clark, C. Farthing and A. Sims,
{\it Simplicity of algebras associated to \'{e}tale groupoids},
Semigroup Forum \textbf{88}(2) (2014), 433--452.

\bibitem{exelbuss}
A. Buss and R. Exel,
{\it Inverse semigroup expansions and their actions on C*-algebras},
Illinois J. Math. \textbf{56}(4) (2012), 1185--1212.

\bibitem{toke}
T. M. Carlsen and N. S. Larsen,
{\it Partial actions and KMS states on relative graph C*-algebras},
J. Funct. Anal. \textbf{271}(8) (2016), 2090--2132.

\bibitem{clarkediemichell}
L. O. Clark and C. Edie-Michell,
{\it Uniqueness theorems for Steinberg algebras},
Algebr. Represent. Theory \textbf{18}(4) (2015), 907--916.

\bibitem{Clark2}
L. O. Clark, C. Edie-Michell, A. an Huef, and A. Sims,
{\it Ideals of Steinberg algebras of strongly effective groupoids, with applications to Leavitt path algebras},
arXiv:1601.07238 [math.RA]

\bibitem{clarkexelpardo}
L. O. Clark, R. Exel and E. Pardo,
{\it A generalised uniqueness theorem and the graded ideal structure of Steinberg algebras},
Forum\ Math., to appear.
arXiv:1609.02873 [math.RA]

\bibitem{galera}
L. O. Clark, R. Exel and E. Pardo, A. Sims, C. Starling,
{\it Simplicity of algebras associated to non-Hausdorff groupoids},
arXiv:1806.04362 [math.OA]

\bibitem{clarketal}
L. O. Clark, C. Farthing, A. Sims and M. Tomforde,
{\it A groupoid generalisation of Leavitt path algebras},
Semigroup Forum \textbf{89}(3) (2014), 501--517.

\bibitem{clarksims}
L. O. Clark and A. Sims,
{\it Equivalent groupoids have Morita equivalent Steinberg algebras},
J. Pure Appl. Algebra \textbf{219}(6) (2015), 2062--2075.

\bibitem{Misha}
M. Dokuchaev,
{\it Recent developments around partial actions}, 
arXiv:1801.09105v2 [math.RA].

\bibitem{dokuchaevexel}
M. Dokuchaev and R. Exel,
{\it Associativity of crossed products by partial actions, enveloping actions and partial representations},
Trans. Amer. Math. Soc. \textbf{357}(5) (2005), 1931--1952.

\bibitem{Dokuchaev}
M. Dokuchaev and R. Exel,
{\it The ideal structure of algebraic partial crossed products},
Proc. Lond. Math. Soc., to appear.
arXiv:1605.07540 [math.OA]

\bibitem{exel1994a}
R. Exel,
{\it Circle actions on C*-algebras, partial automorphisms, and a generalized Pimsner-Voiculescu exact sequence}, 
J. Funct. Anal. \textbf{122}(2) (1994), 361--401.



\bibitem{exel1998}
R. Exel,
{\it Partial actions of groups and actions of inverse semigroups},
Proc. Amer. Math. Soc. \textbf{126}(12) (1998), 3481--3494.


\bibitem{exellaca}
R. Exel and M. Laca,
{\it Cuntz-Krieger algebras for infinite matrices},
J. Reine Angew. Math. \textbf{512} (1999), 119--172.

\bibitem{ExelPardo}
R.Exel and E. Pardo
{\it The tight groupoid of an inverse semigroup}
Semigroup Forum \textbf{92}(1) (2016), 274--303.

\bibitem{exelvieira}
R. Exel and F. Vieira,
{\it Actions of inverse semigroups arising from partial actions of groups},
J. Math. Anal. Appl. \textbf{363}(1) (2010), 86--96.




\bibitem{GGS}
T. Giordano, D. Gon\c{c}alves and C. Starling,
{\it Bratteli-Vershik models for partial actions of $\mathbb{Z}$},
to appear in Internat. J. Math., arXiv:1611.04389v2.

\bibitem{Gonc}
D. Gon\c{c}alves,
{\it Simplicity of partial skew group rings of abelian groups},
Canad. Math. Bull. \textbf{57}(3) (2014), 511--519.

\bibitem{GOR}
D. Gon\c{c}alves, J. \"{O}inert and D. Royer,
{\it  Simplicity of partial skew group rings with applications to Leavitt path algebras and topological dynamics},
J. Algebra \textbf{420} (2014), 201--216.

\bibitem{GRHouston}
D. Gon\c{c}alves and D. Royer,
{\it C*-algebras associated to stationary ordered Bratteli diagrams},
Houston J. Math. \textbf{40}(1) (2014), 127--143.

\bibitem{goncalvesroyer}
D. Gon\c{c}alves and D. Royer,
{\it Leavitt path algebras as partial skew group rings},
Comm. Algebra \textbf{42}(8) (2014), 3578--3592.

\bibitem{GRUltra}
D. Gon\c{c}alves and D. Royer,
{\it Ultragraphs and shift spaces over infinite alphabets},
Bull. Sci. Math. {\bf 141}(1) (2017), 25--45.

\bibitem{GRIRMN}
D. Gon\c{c}alves and D. Royer,
{\it Infinite alphabet edge shift spaces via ultragraphs and their C*-algebras},
to appear in Int. Math. Res. Not. IMRN.
arXiv:1703.05069 [math.OA]

\bibitem{GRUP}
D. Gon\c{c}alves and D. Royer,
{\it Simplicity and chain conditions for ultragraph Leavitt path algebras via partial skew group ring theory},
arXiv:1706.03628v2.


\bibitem{MR0262402}
W. D. Munn,
{\it Fundamental inverse semigroups},
Quart. J. Math. Oxford Ser. (2)  {\bf (21)} (1970), 157--170.


\bibitem{OinertArt}
P. Nystedt, J. Öinert and H. Pinedo,
{\it Artinian and noetherian partial skew groupoid rings},
arXiv:1603.02237v2.

\bibitem{steinberg}
B. Steinberg,
{\it A groupoid approach to discrete inverse semigroup algebras},
Adv.\ Math. \textbf{223}(2) (2010), 689--727.

\bibitem{steinberg2014}
B. Steinberg,
{\it Modules over \'{e}tale groupoid algebras as sheaves},
J. Aust. Math. Soc. \textbf{97}(3) (2014), 418--429.

\bibitem{Ben}
B. Steinberg,
{\it Simplicity, primitivity and semiprimitivity of étale groupoid algebras with applications to inverser semigroup algebras},
J. Pure Appl. Algebra \textbf{220} (2016), 1035--1054.



\end{thebibliography}
\end{document}